\documentclass[final, reqno, 11pt]{amsart}
\usepackage{amsmath}
\usepackage{amsfonts}
\usepackage{amssymb}
\usepackage{amsthm}

\newtheorem{theorem}{Theorem}[section]
\newtheorem{lemma}[theorem]{Lemma}
\newtheorem{proposition}[theorem]{Proposition}
\newtheorem{corollary}[theorem]{Corollary}
\theoremstyle{remark}

\newtheorem{definition}[theorem]{Definition}

\newcommand{\dl}{\nabla}
\newcommand{\les}{\lesssim}
\newcommand{\ges}{\gtrsim}

\newcommand{\mc}{\mathcal}
\newcommand{\be}{\begin{equation}}
\newcommand{\ee}{\end{equation}}
\newcommand{\ba}{\begin{array}}

\newcommand{\ea}{\end{array}}
\newcommand{\bpm}{\begin{pmatrix}}
\newcommand{\epm}{\end{pmatrix}}
\newcommand{\lb}{\label}

\DeclareMathOperator{\supp}{supp}

\newcommand{\ov}{\overline}
\newcommand{\dd}{{\,}{d}}

\newcommand{\R}{\mathbb R}

\title[Large Initial Data]{Large Initial Data Global Well-Posedness for a Supercritical Wave Equation}
\author{Marius Beceanu}
\address{University at Albany SUNY, Department of Mathematics and Statistics, Earth Science 110, Albany, NY, 12222, USA}
\email{mbeceanu@albany.edu}
\author{Avy Soffer}
\address{Rutgers University Department of Mathematics, 110 Frelinghuysen Rd., Piscataway, NJ, 08854, USA}
\email{soffer@math.rutgers.edu}
\begin{document}
\maketitle
\numberwithin{equation}{section}
\begin{abstract} We prove the existence of global solutions to the focusing energy-supercritical semilinear wave equation in $\R^{3+1}$ for arbitrary outgoing large initial data, after we modify the equation by projecting the nonlinearity on outgoing states.
\end{abstract}

\tableofcontents
\section{Introduction}
Consider the focusing semilinear wave equation on~$\R^{3+1}$
\be\lb{defocusing}
u_{tt}-\Delta u-|u|^N u =0,\ u(0)=u_0,\ u_t(0)=u_1.
\ee
This equation is $\dot H^{s_c}$-critical for $s_c=\frac 3 2 - \frac 2 N$, making it energy-supercritical for $N>4$.

An equivalent formulation of equation (\ref{defocusing}) is
$$
u(t)=\cos(t\sqrt{-\Delta})u_0 + \frac{\sin(t\sqrt{-\Delta})}{\sqrt{-\Delta}}u_1 + \int_0^t \frac{\sin((t-s)\sqrt{-\Delta})}{\sqrt{-\Delta}} (|u(s)|^Nu(s)) \dd s
$$
or in other words
$$
u(t)=\Phi_0(t)(u_0, u_1)+\int_0^t \Phi_0(t-s)(0, |u(s)|^Nu(s)) \dd s,
$$
where
$$\begin{aligned}
&\Phi(t)(u_0, u_1) = (\Phi_0(t)(u_0, u_1), \Phi_1(t)(u_0, u_1)) \\
&:= (\cos(t\sqrt{-\Delta}u_0 + \frac{\sin(t\sqrt{-\Delta})}{\sqrt{-\Delta}}u_1, -\sin(t\sqrt{-\Delta})\sqrt{-\Delta}u_0+\cos(t\sqrt{-\Delta})u_1)
\end{aligned}$$
is the flow of the linear wave equation in three dimensions.

Since this equation is hard to analyze, we consider instead a simpler model, where we project the nonlinearity on outgoing states first:
\be\lb{eq_outgoing}
u(t)=\Phi_0(t)(u_0, u_1)+\int_0^t \Phi_0(t-s)P_+(0, |u(s)|^Nu(s)) \dd s.
\ee
Here $P_+$ is the projection on outgoing states, see (\ref{outgoing}). In Section \ref{derivation} we obtain several more concrete equivalent formulations of equation (\ref{eq_outgoing}).

Equation (\ref{eq_outgoing}) has the same scaling as equation (\ref{defocusing}) and in particular is energy-supercritical for $N>4$.

All concrete formulations of equation (\ref{eq_outgoing}) --- see (\ref{eq_out}), (\ref{equivalent}), (\ref{eq_outgoing'}), and (\ref{eqn}) --- involve the nonlocal operator
$$
f(r) \mapsto \frac 1 r \int_0^r \rho f(\rho) \dd \rho.
$$
Therefore, the solution has infinite propagation speed and even if it starts with compact support it immediately extends to the whole of $\R^3$. In general, at times $t>0$ the best decay rate we can expect is $1/|x|$.

The reason why we take the focusing sign in equation (\ref{defocusing}) is that the projection on outgoing states $P_+$ changes the sign of the nonlinearity, so equation (\ref{eq_outgoing}) is in fact defocusing.

For equation (\ref{eq_outgoing}), it turns out that all sufficiently regular and outgoing (in the sense of Definition \ref{def_1}) initial data lead to global solutions forward in time, even when the equation is energy-supercritical, i.e.\ $N>4$.

We begin with a global existence result for finite energy bounded initial data.

\begin{theorem}\lb{main_thm} Assume $N \geq 2$ and consider radial outgoing (in the sense of Definition \ref{def_1}) initial data $(u_0, u_1) \in ((\dot H^1 \cap L^\infty) \times L^2)_{out}$. Then the corresponding solution $u$ to equation (\ref{eq_outgoing}) exists globally on $\R^3 \times [0, \infty)$ and
$$
\|u\|_{L^\infty_t \dot H^1_x} \les \|u_0\|_{\dot H^1},\ \|u(t)\|_{L^\infty} \les t^{-1/2} \|u_0\|_{\dot H^1},
$$
and
$$
\|u\|_{L^\infty_{t, x}} \les \|u_0\|_{L^\infty} + \|u_0\|_{\dot H^1}^{2} \|u_0\|_{L^\infty}^{(N-2)/2} + \|u_0\|_{\dot H^1} \|u_0\|_{L^\infty}^{N/2}.
$$
In addition, assuming that $N>4$,
$$\begin{aligned}
\|u\|_{L^2_t L^\infty_x} &\les \|u_0\|_{L^\infty} + (\ln_+ \|u_0\|_{\dot H^1} + \ln_+ \|u_0\|_{L^\infty} + 1)^{1/2} \|u_0\|_{\dot H^1}.
\end{aligned}$$
\end{theorem}

We next state a low regularity global existence result that holds for $L^{N+2}$ initial data supported away from the origin.
\begin{theorem}\lb{Nexistence} Assume $N>2$ and consider radial and outgoing initial data $(u_0, u_1)$ with $u_0 \in L^{N+2}$ and $\supp u_0 \subset \ov{\R^3 \setminus B(0, R)}$ with $R>0$. Then the corresponding solution $u$ of (\ref{eq_outgoing}) exists globally on $\R^3 \times [0, \infty)$ and
\be\lb{uN}
\|u\|_{L^\infty_t L^{N+2}_x} \leq \|u_0\|_{L^{N+2}}.
\ee
In addition, if $N>4$ then $\|u(t)\|_{L^{N+2}} \to 0$ as $t \to \infty$.
\end{theorem}
Note that the solution $u(t)$ at time $t$ is supported on $\ov{\R^3\setminus B(0, R+t)}$. Since $\|u\|_{L^{N+2}}=\|r^{2/(N+2)} u(r)\|_{L^{N+2}_r([0, \infty))}$, we can rephrase inequality (\ref{uN}) as
$$
\|u(t)\|_{L^{N+2}_r([0,\infty))} \les (R+t)^{-2/(N+2)} \|u_0\|_{L^{N+2}}.
$$
This means that the solution disperses in a certain sense.

When dealing with data supported near the origin, we need more regularity even for local existence. As a corollary to Theorem \ref{Nexistence}, we obtain that a global solution always exists for $\dot H^{s_c} \times \dot H^{s_c-1}$ outgoing initial data.

\begin{corollary}\lb{optimal_existence} Assume that $N>2$ and $(u_0, u_1) \in (\dot H^{s_c} \cap L^{N+2} \times \dot H^{s_c-1})_{out}$ are radial and outgoing. Then the corresponding solution $u$ of (\ref{eq_outgoing}) exists globally on $\R^3 \times [0, \infty)$ and
$$
\|u\|_{L^\infty_t L^{N+2}_x} \leq \|u_0\|_{L^{N+2}}.
$$
In addition, if $N>4$ then $\|u(t)\|_{L^{N+2}} \to 0$ as $t \to \infty$.
\end{corollary}
This hypothesis is optimal from a scaling point of view, but certainly not optimal in terms of the number of derivatives. For outgoing solutions one has the Strichartz-type estimate
$$
\|\Phi_0(u_0, u_1)\|_{L^{N/2}_t L^\infty_x} \les \|u_0\|_{|x|^{-2/N-\epsilon}L^\infty_x \cap |x|^{-2/N+\epsilon}L^\infty_x}.
$$
Therefore we could take $|x|^{-2/N-\epsilon}L^\infty_x \cap |x|^{-2/N+\epsilon}L^\infty_x \cap L^{N+2}$ initial data.

Previous papers on the topic of supercritical wave equations include \cite{becsof}, \cite{bul1}, \cite{bul2}, \cite{bul3}, \cite{dkm}, \cite{kivi1}, \cite{kivi2}, \cite{krsc}, \cite{roy}, \cite{roy2}, \cite{struwe}, and \cite{tao}. For more details, the reader is referred to \cite{becsof}.

Our results are in keeping with the principle that, for radially symmetric solutions, blow-up can only occur at the origin, so it is precluded in our case due to the outgoing character of the equation (\ref{eq_outgoing}). In fact, for this reason, a higher power of the nonlinearity makes the equation easier to solve, since it means more decay.

The proofs are based on local existence results and on two scaling-subcritical conservation laws that we leverage in order to control the solution.

The hypothesis that the initial data are outgoing is in fact not necessary. This will be addressed in a future version of this paper.

Equation (\ref{eq_outgoing}) is a concrete example of an energy-supercritical dispersive equation, with no scaling-critical conserved quantities, that can be completely solved for arbitrary large initial data. In addition, studying this simplified model may lead to new insight concerning the original equation~(\ref{defocusing}).

The paper is organized as follows: in Section \ref{preliminary} we recall the definition of incoming and outgoing states and state some linear estimates. In Section \ref{derivation} we derive some alternative formulations of equation (\ref{eq_outgoing}). In Section \ref{local_ex} we prove some local existence results and small data global existence results. In Section \ref{conservation_laws} we show some conservation laws for the equation. Finally, in Section \ref{proof_main} we prove the main results stated in the introduction.

\section{Notations}
$A \les B$ means that $|A| \leq C |B|$ for some constant $C$. We denote various constants, not always the same, by $C$.

The Laplacian is the operator on $\R^3$ $\Delta=\frac {\partial^2}{\partial_{x_1}^2} + \frac {\partial^2}{\partial_{x_2}^2} + \frac {\partial^2}{\partial_{x_3}^2}$.

We denote by $L^p$ the Lebesgue spaces, by $\dot H^s$ and $\dot W^{s, p}$ (fractional) homogenous Sobolev spaces, and by $L^{p, q}$ Lorentz spaces. We also define the weighted Lebesgue spaces $w(x) L^p_x := \{w(x) f(x): f \in L^p\}$.

$\dot H^s$ are Hilbert spaces and so is $\dot H^1 \times L^2$, with the dot product
$$
\langle (u_0, u_1), (U_0, U_1) \rangle_{\dot H^1 \times L^2} = \int_{\R^3} \dl u_0(x) \dl \ov U_0(x) + u_1(x) \ov U_1(x) \dd x.
$$

For a radially symmetric function $u(x)$, we let $u(r):=u(x)$ for $|x|=r$.

By $(\dot H^1 \times L^2)_{out}$ we mean the space of outgoing radially symmetric $\dot H^1 \times L^2$ initial data, see Definition \ref{def_1}.

We define the mixed-norm spaces on $\R^3 \times [0, \infty)$
$$
L^p_t L^q_x := \Big\{f \mid \|f\|_{L^p_t L^q_x}:= \Big(\int_0^\infty \|f(x, t)\|_{L^q_x}^p \dd t\Big)^{1/p} < \infty \Big\},
$$
with the standard modification for $p=\infty$. 
Also, for $I \subset [0, \infty)$, let $\|f\|_{L^p_t L^q_x(\R^3 \times I)} := \|\chi_I(t) f\|_{L^p_t L^q_x}$, where $\chi_I$ is the characteristic function of~$I$.

We also denote $B(0, R):=\{x \in \R^3 \mid |x| \leq R\}$.


Let $\Phi(t):\dot H^1 \times L^2 \to \dot H^1 \times L^2$ be the flow of the linear wave equation in three dimensions: for
$$
u_{tt}-\Delta u=0,\ u(0)=u_0,\ u_t(0)=u_1,
$$
we set $\Phi(t)(u_0, u_1)=(\Phi_{0}(t)(u_0, u_1), \Phi_{1}(t)(u_0, u_1)):=(u(t), u_t(t))$.

\section{Preliminary estimates}\lb{preliminary}
Recall the standard Strichartz estimates for the free wave equation in three dimensions, see \cite{give}, \cite{keeltao}, and \cite{klma}:
\begin{lemma} Consider a solution $u$ of the free wave equation in dimension three:
$$
u_{tt}-\Delta u =0,\ u(0)=u_0,\ u_t(0)=u_1.
$$
If $(u_0, u_1) \in \dot H^1 \times L^2$ are radial, then
$$
\|u\|_{L^\infty_t \dot H^1_x \cap L^4_t \dot W^{1/2, 4}_x \cap L^2_t L^\infty_x} \les \|u_0\|_{\dot H^1} + \|u_1\|_{L^2}.
$$
\end{lemma}
Note that in particular, as shown in \cite{klma}, the endpoint $L^2_t L^\infty_x$ Strichartz estimate is true in the radial case.


One of the main technical tools we use in this paper is a decomposition of all radial initial data into incoming and outgoing states. Intuitively, incoming states are those that move toward the origin, while outgoing states are the ones that move away from the origin.

To define incoming and outgoing states, we use the following lemma borrowed from \cite{becsof}:
\begin{lemma} There exist bounded operators $P_+$ and $P_-$ on $\dot H^1_{rad} \times L^2_{rad}$, given by
\be\lb{outgoing}\begin{aligned}
P_+(u_0, u_1) &=(P_{0+}(u_0, u_1), P_{1+}(u_0, u_1))\\
&:= \Big(\frac 1 2 \Big(u_0 - \frac 1 r \int_0^r \rho u_1(\rho) \dd \rho\Big), \frac 1 2 \big(-(u_0)_r-\frac {u_0} r + u_1\big)\Big)
\end{aligned}\ee
and
\be\lb{incoming}\begin{aligned}
P_-(u_0, u_1) &=(P_{0-}(u_0, u_1), P_{1-}(u_0, u_1))\\
&:= \Big(\frac 1 2 \Big(u_0 + \frac 1 r \int_0^r \rho u_1(\rho) \dd \rho\Big), \frac 1 2 \big((u_0)_r+\frac {u_0} r + u_1\big)\Big),
\end{aligned}\ee
such that $I=P_++P_-$, $P_+^2=P_+$, and $P_-^2=P_-$.
	
If $\Phi(t)$ is the flow of the linear equation then for $t \geq 0$ $P_- \Phi(t) P_+=0$ and for $t \leq 0$ $P_+ \Phi(t) P_-=0$. In addition, for $t \geq 0$ $\Phi(t)P_+(u_0, u_1)$ is supported on $\ov{\R^3\setminus B(0, t)}$ and for $t \leq 0$ $\Phi(t)P_-(u_0, u_1)$ is supported on $\ov{\R^3 \setminus B(0, -t)}$.
\end{lemma}

In addition, we use the following related fact: suppose that $P_+(u_0, u_1)=(u_0, u_1)$ and $(u_0, u_1)$ are supported on $\ov {\R^3\setminus B(0, R)}$; then $\Phi(t)(u_0, u_1)$ is supported on $\ov {\R^3\setminus B(0, R+t)}$ for $t \geq -R$.

\begin{definition}\lb{def_1} $P_+$ and $P_-$ are called the projection on outgoing, respectively incoming states. We call any radial $(u_0, u_1)$ such that $P_-(u_0, u_1)=0$ \emph{outgoing}; if $P_+(u_0, u_1)=0$ we call it \emph{incoming}.
\end{definition}

We next recall another lemma from \cite{becsof}, concerning the nonlocal operator that enters the definition of $P_+$ and $P_-$.
\begin{lemma}\lb{bounds} For radial $f \in L^2$
$$
\Big\|\frac 1 r \int_0^r \rho f(\rho) \dd \rho\Big\|_{\dot H^1_{rad}} \les \|f\|_{L^2_{rad}}.
$$
More generally, for $0 \leq s < 3/2$
\be\lb{embedding}
\Big\|\frac 1 r \int_0^r \rho f(\rho) \dd \rho\Big\|_{\dot H^{s+1}_{rad}} \les \|f\|_{\dot H^s_{rad}}.
\ee
Consequently, $P_+$ and $P_-$ are bounded on $\dot H^s \times \dot H^{s-1}$ for $1 \leq s < 3/2$.

Furthermore, if $(u_0, u_1)$ are purely outgoing or purely incoming, then $\|u_0\|_{\dot H^s} \sim \|u_1\|_{\dot H^{s-1}}$ for $1 \leq s < 3/2$.
\end{lemma}

We need one more property of this nonlocal operator.
\begin{lemma}\lb{nonlocal} Consider $f \in L^2_{rad} \cap L^\infty$. Then
$$
\Big\|\frac 1 r \int_0^r \rho f(\rho) \dd \rho\Big\|_{L^\infty} \les \|f\|_{L^2_{rad} \cap L^\infty}.
$$
\end{lemma}
\begin{proof}
Clearly
$$
\Big|\frac 1 r \int_0^r \rho f(\rho) \dd \rho\Big| \les \int_0^r |f(\rho)| \dd \rho.
$$
Then
$$\begin{aligned}
&\int_0^r |f(\rho)| \dd \rho \leq \int_0^1 |f(\rho)| \dd \rho + \int_1^\infty |f(\rho)| \dd \rho \\
&\les \|f\|_{L^\infty} + \|f(\rho) \rho\|_{L^2([1, \infty))} \|1/\rho\|_{L^2([1, \infty))} \les \|f\|_{L^\infty} + \|f\|_{L^2_{rad}}.
\end{aligned}$$
\end{proof}

This inequality can be improved:
\begin{lemma}\lb{improved}
Consider a radial function $f$ on $\R^3$. Then
\be\lb{lp_bound}
\Big\|\frac 1 r \int_0^r \rho f(\rho) \dd \rho\Big\|_{L^\infty} \les \|f\|_{L^{3, 1}},\ \Big\|\frac 1 r \int_0^r \rho f(\rho) \dd \rho\Big\|_{L^{3, \infty}} \les \|f\|_{L^{3/2, 1}},
\ee
and for $3/2<p<3$
$$
\Big\|\frac 1 r \int_0^r \rho f(\rho) \dd \rho\Big\|_{L^{3p/(3-p)}} \les \|f\|_{L^p}.
$$
\end{lemma}
This estimate means there is no need to take $L^\infty$ initial data in our equation (\ref{eq_outgoing}) --- $L^p$ with $p$ sufficiently large is sufficient --- but for simplicity we choose not to pursue this idea in this paper.
\begin{proof}[Proof of Lemma \ref{improved}] We rewrite this nonlocal operator as
$$
\frac 1 {4\pi|x|} \int_{|y| \leq |x|} \frac {f(y)}{|y|} \dd y \leq \frac 1 {4\pi} \int_{|y| \leq |x|} \frac {|f(y)|}{|y|^2} \dd y.
$$

The two estimates (\ref{lp_bound}) follow immediately from these representations and from the fact that $\frac 1 {|y|} \in L^{3, \infty}$ (which pairs with $L^{3/2, 1}$) and $\frac 1 {|y|^2} \in L^{3/2, \infty}$ (which pairs with $L^{3, 1}$).

The remaining estimate follows by real interpolation (see Theorem 5.3.1 in \cite{bergh}). Iin fact, $L^{3p/(3-p)}$ can be further improved to $L^{3p/(3-p), p}$.
\end{proof}

We are also interested in estimates that hold only for $f$ supported away from zero.
\begin{lemma}\lb{awayfromzero} Consider a radial function $f$ on $\R^3$ such that $\supp f \subset \ov{\R^3 \setminus B(0, R)}$, where $R>0$, and suppose $1 \leq p \leq 2$. Then
$$
\Big\|\frac 1 r \int_0^r \rho f(\rho) \dd \rho \Big\|_{L^\infty} \les R^{1-3/p} \|f\|_{L^p},\ \Big\|\frac 1 r \int_0^r \rho f(\rho) \dd \rho \Big\|_{L^{3p, \infty}} \les R^{1-2/p} \|f\|_{L^p},
$$
and for $3p<q \leq \infty$
$$
\Big\|\frac 1 r \int_0^r \rho f(\rho) \dd \rho \Big\|_{L^q} \les R^{1-3/p+3/q} \|f\|_{L^p}.
$$
\end{lemma}
\begin{proof}[Proof of Lemma \ref{awayfromzero}] Since $f$ is radial, $f \in L^p$ is equivalent to $\rho^{2/p} f(\rho) \in L^p([0, \infty))$. Then
$$\begin{aligned}
& \frac 1 r \int_0^r \rho f(\rho) \dd \rho \les \frac {R^{1-2/p}} r \int_0^r \rho^{2/p} |f(\rho)| \dd \rho \\
& \les \frac {R^{1-2/p}} r (r-R)^{1-1/p} \|\rho^{2/p} f(\rho)\|_{L^p([0, \infty))} \les R^{1-2/p} r^{-1/p} \|f\|_{L^p}.
\end{aligned}$$
The first two conclusions then follow immediately; note that $r^{-1/p} \in L^{3p, \infty}$ and that $r^{-1/p} \leq R^{-1/p}$ on the domain of $f$. The third conclusion then follows by interpolation.
\end{proof}

We next prove a simple dispersive estimate for outgoing solutions.
\begin{lemma} Let $u$ be the solution of the linear wave equation
$$
u_{tt}-\Delta u =0,\ u(0)=u_0,\ u_t(0)=u_1,
$$
with outgoing initial data $(u_0, u_1)$. Then for $0 \leq s < 3/2$ $\|u(x, t)\|_{\dot H^s_x} \les \|u_0\|_{\dot H^s_x}$ and for $1/2 \leq s<3/2$ and $t \geq 0$ $\|u(x, t)\|_{L^\infty_x} \les t^{s-3/2} \|u_0\|_{\dot H^s}$.
\end{lemma}
The endpoint $t^{-1}$ decay can also be achieved by e.g.\ using Besov spaces and interpolation (or, more simply, the inhomogenous $H^1$ norm). More interestingly, the $\dot H^s$ norms can be replaced by weighted $L^\infty$ norms.
\begin{proof} The first inequality follows because
$$
\|u(x, t)\|_{\dot H^s_x} \les \|u_0\|_{\dot H^s} + \|u_1\|_{\dot H^{s-1}} \les \|u_0\|_{\dot H^s}
$$
by Hardy's inequality (since $u_1=(u_0)_r+\frac{u_0}r$). The second inequality follows by the radial Sobolev embedding
$$
|u(r)| \les r^{s-3/2} \|u\|_{\dot H^s}
$$
because at time $t \geq 0$ the solution $u$ is supported on $\ov{\R^3\setminus B(0, t)}$.
\end{proof}

We now state some special identities that hold for outgoing solutions only, which show the improvements that occur compared to the general case.
\begin{lemma} Assume that $u$ is a smooth, compactly supported, and outgoing solution to the linear wave equation for $t \geq 0$:
$$
u_{tt}-\Delta u =0,\ u(0)=u_0,\ u_t(0)=u_1.
$$
Then
\be\lb{ident}
\int_{\R^3 \times \{T\}} |u|^n \dd x = \int_{\R^3} |u_0|^n \dd x - (n-2) \int_0^T \int_{\R^3} \frac {|u|^n}{|x|} \dd x \dd t.
\ee
\end{lemma}
In particular, for $n>2$
$$
\int_{\R^3 \times \R} \frac {|u|^n}{|x|} \dd x \dd t \les \|u_0\|_{L^n}^n.
$$
\begin{proof}
The solution $u$ being outgoing means that $u_t+u_r+u/r=0$. Note that
$$
\frac d {dt} \int_{\R^3\times\{t\}} |u|^n \dd x = \int_{\R^3 \times \{t\}} n |u|^{n-2} u u_t \dd x = -\int_{\R^3 \times \{t\}} n |u|^{n-2} u u_r + n \frac{|u|^n}{|x|} \dd x
$$
and
\be\lb{parts}\begin{aligned}
&\int_{\R^3 \times \{t\}} n |u|^{n-2} u u_r \dd x = 4\pi \int_0^\infty n |u|^{n-2} u u_r r^2 \dd r \\
&= 4\pi (|u|^n r^2) \mid_0^\infty - 4\pi \int_0^\infty |u|^n 2r \dd r = -2 \int_{\R^3 \times\{t\}} \frac {|u|^n}{|x|} \dd x.
\end{aligned}\ee
Integrating in $t$ we obtain (\ref{ident}). The other conclusion is now obvious.
\end{proof}

Finally, we summarize some simple, but important results from \cite{becsof}.
\begin{proposition}\lb{est_outgoing} Consider a solution $u$ to the linear wave equation on $\R^{3+1}$ with outgoing initial data $(u_0, u_1)$:
$$
u_{tt}-\Delta u=0,\ u(0)=u_0,\ u_t(0)=u_1.
$$
Then for $0 \leq t \leq r$
\be\lb{sol_outgoing}
u(r, t)=\frac {r-t} r u_0(r-t)
\ee
and $u(r, t)=0$ for $0 \leq r \leq t$. Therefore
\be\lb{linf}
\|u\|_{L^\infty_{t, x}} \leq \|u_0\|_{L^\infty}.
\ee
More generally, for $2 \leq p \leq \infty$ $\|u\|_{L^\infty_t L^p_x} \leq \|u_0\|_{L^p}$.
\end{proposition}
Note that actually for $2<p< \infty$, by dominated convergence, $\|u(t)\|_{L^p} \to 0$ as $t \to \infty$. Another easy consequence of (\ref{sol_outgoing}) is that
\be\lb{linfw}
\|u\|_{|x|^{-1} L^\infty_{t, x}} \leq \|u_0\|_{|x|^{-1} L^\infty_x}.
\ee

\section{Derivation of the equation}\lb{derivation}
We next perform a rigorous derivation of several alternative formulations for equation (\ref{eq_outgoing}):
$$
u(t)=\Phi_0(t)(u_0, u_1)+\int_0^t \Phi_0(t-s)P_+(0, |u(s)|^Nu(s)) \dd s,
$$
where $P_+$ is the projection on outgoing states, see (\ref{outgoing}).

%
%

By taking a derivative in $t$ we obtain a similar equation for $u_t$, namely
\be\lb{eq_ut}
u_t(t)=\Phi_1(t)(u_0, u_1) + P_{0+}(0, |u(t)|^Nu(t)) + \int_0^t \Phi_1(t-s)P_+(0, |u(s)|^Nu(s)) \dd s,
\ee
where the extra term comes from the derivative hitting the integral. Here
$$
P_{0+}(u_0, u_1):=\frac 1 2 \Big(u_0-\frac 1 r \int_0^r\rho u_1(\rho) \dd \rho\Big)
$$
is the first component of $P_+$.

We next write equation (\ref{eq_outgoing}) in a more explicit form. We start with
$$\begin{aligned}
&P_+(0, |u(s)|^Nu(s))=\Big(-\frac 1 {2r} \int_0^r \rho |u(\rho, s)|^Nu(\rho, s) \dd \rho, \frac 1 2 |u(s)|^Nu(s)\Big).
\end{aligned}$$
Then
\be\lb{com}\begin{aligned}
&\int_0^t \Phi_0(t-s)P_+(0, |u(r, s)|^Nu(r, s)) \dd s = \\
&= \frac 1 2 \Big(-\int_0^t \cos((t-s)\sqrt{-\Delta}) \Big(\frac 1 r \int_0^r \rho |u(\rho, s)|^Nu(\rho, s) \dd \rho\Big) \dd s + \\
&+ \int_0^t \frac {\sin((t-s)\sqrt{-\Delta})}{\sqrt{-\Delta}} |u(s)|^Nu(s) \dd s\Big) \\
&= -\frac 1 2 \frac{\sin(t\sqrt{-\Delta})}{\sqrt{-\Delta}} \Big(\frac 1 r \int_0^r \rho |u_0(\rho)|^Nu_0(\rho) \dd \rho\Big) + \\
&+ \frac 1 2 \int_0^t \frac{\sin((t-s)\sqrt{-\Delta})}{\sqrt{-\Delta}}\Big(-\frac {N+1} r \int_0^r \rho |u(\rho, s)|^Nu_t(\rho, s) \dd \rho + |u(s)|^Nu(s)\Big) \dd s.
\end{aligned}\ee
Thus the equation (\ref{eq_outgoing}) becomes
\be\lb{eq_out}\begin{aligned}
&u(t) = \cos(t\sqrt{-\Delta}) u_0 + \frac {\sin(t\sqrt{-\Delta})}{\sqrt{-\Delta}} u_1 - \frac 1 2 \int_0^t \cos((t-s)\sqrt{-\Delta}) \\
&\Big(\frac 1 r \int_0^r \rho |u(\rho, s)|^Nu(\rho, s) \dd \rho\Big) \dd s + \frac 1 2 \int_0^t \frac {\sin((t-s)\sqrt{-\Delta})}{\sqrt{-\Delta}} |u(s)|^Nu(s) \dd s.
\end{aligned}\ee
As shown by (\ref{com}), another equivalent formulation of (\ref{eq_outgoing}) is
\be\lb{equivalent}\begin{aligned}
&u_{tt}-\Delta u + \frac 1 2 \Big(\frac {N+1} r \int_0^r \rho |u(\rho)|^Nu_t(\rho) \dd \rho - |u|^Nu\Big) = 0, \\
&u(0)=u_0,\ u_t(0)=-\frac 1 {2r} \int_0^r \rho |u_0(\rho)|^Nu_0(\rho) \dd \rho + u_1.
\end{aligned}\ee
Note that the original equation (\ref{defocusing}) is equivalent to the system
$$
\left\{\begin{array}{l}
u_t=v \\
v_t=\Delta u + |u|^Nu
\end{array}\right.,\ u(0)=u_0,\ v(0)=u_1.
$$
\begin{lemma}\lb{lemma_system}
The modified equation (\ref{eq_outgoing}) is equivalent to the system
\be\lb{eq_outgoing'}
\left\{\begin{array}{l}
\displaystyle u_t=v-\frac 1 {2r} \int_0^r \rho |u(\rho)|^Nu(\rho) \dd \rho \\
\displaystyle v_t=\Delta u + \frac 1 2 |u|^N u 
\end{array}\right.,\ u(0)=u_0,\ v(0)=u_1.
\ee
\end{lemma}
\begin{proof}[Proof of Lemma \ref{lemma_system}] Assume that (\ref{eq_outgoing'}) holds. Setting $t=0$ we obtain that
$$
u_t(0)=v(0)-\frac 1 {2r} \int_0^r \rho |u_0(\rho)|^Nu_0(\rho) \dd \rho=u_1-\frac 1 {2r} \int_0^r \rho |u_0(\rho)|^Nu_0(\rho) \dd \rho
$$
and by differentiating we obtain
$$\begin{aligned}
u_{tt}&=v_t-\frac {N+1} {2r} \int_0^r \rho |u(\rho)|^N u_t(\rho) \dd \rho \\
&=\Delta u + \frac 1 2 |u|^N u -\frac {N+1} {2r} \int_0^r \rho |u(\rho)|^N u_t(\rho) \dd \rho,
\end{aligned}$$
that is (\ref{equivalent}).
	
Conversely, assume that (\ref{equivalent}) holds. Let $v=u_t+\frac 1 {2r} \int_0^r \rho |u(\rho)|^Nu(\rho) \dd \rho$. Then
$$
v(0)=u_t(0) + \frac 1 {2r} \int_0^r \rho |u_0(\rho)|^Nu_0(\rho) \dd \rho = u_1
$$
and
$$\begin{aligned}
v_t&=u_{tt} + \frac {N+1} {2r} \int_0^r \rho |u(\rho)|^N u_t(\rho) \dd \rho =\Delta u + \frac 1 2 |u|^N u.
\end{aligned}$$
Thus $(u, v)$ satisfy (\ref{eq_outgoing'}).
\end{proof}

Finally, equation (\ref{eq_outgoing'}) can be rephrased as a first-order equation:
\begin{lemma}\lb{lemma_firstorder} For smooth and compactly supported solutions, equation (\ref{eq_outgoing}) is equivalent to
\be\lb{eqn}\begin{aligned}
u_t+u_r+\frac u r+\frac 1 {2r}\int_0^r \rho|u(\rho)|^Nu(\rho) \dd \rho =  (\partial_r + \frac 1 r) \Phi_0(t)(u_0, u_1) + \\
+ \Phi_1(t)(u_0, u_1), u(0)=u_0.
\end{aligned}\ee
In particular, if the initial data $(u_0, u_1)$ are outgoing, the equation becomes
\be\lb{eqn'}
u_t+u_r+\frac u r+\frac 1 {2r}\int_0^r \rho|u(\rho)|^Nu(\rho) \dd \rho = 0,\ u(0)=u_0.
\ee
\end{lemma}

Observe that even if the initial data $(u_0, u_1)$ are outgoing, (\ref{eqn'}) shows that the pair $(u(t), u_t(t))$ fails to be outgoing at any time $t \geq 0$, even at time $t=0$ (also see (\ref{equivalent}), which shows that $(u(0), u_t(0))\ne(u_0, u_1)$).

On the other hand, in view of (\ref{eq_outgoing'}), (\ref{eqn'}) precisely means that $(u(t), v(t))$ are outgoing for $t \geq 0$.

\begin{proof}[Proof of Lemma \ref{lemma_firstorder}] Assume that $u$ solves (\ref{eq_outgoing}); hence (\ref{eq_ut}) also holds.
	
Let $P_{1-}(u_0, u_1):=(u_0)_r+\frac {u_0} r+u_1$ be the second component of the incoming projection $P_-$. Combining (\ref{eq_outgoing}) and (\ref{eq_ut}) and applying $P_{1-}$ to $(u, u_t)$ we obtain
$$\begin{aligned}
P_{1-} (u(t), u_t(t)) &= P_{1-} \Phi(t) (u_0, u_1) + P_{1-} (0, P_{0+}(0, |u(t)|^Nu(t))) + \\
&+ \int_0^t P_{1-} \Phi(t-s) P_+(0, |u(s)|^Nu(s)) \dd s.
\end{aligned}$$
However, $P_{1-} \Phi(t-s) P_+=0$ for $t \geq s$ and
$$
P_{1-}(0, P_{0+}(0, |u(t)|^Nu(t))) = P_{0+}(0, |u(t)|^Nu(t)) = -\frac 1 {2r}\int_0^r \rho|u(\rho, t)|^Nu(\rho, t) \dd \rho.
$$
We obtain exactly (\ref{eqn}).
	
Conversely, assume (\ref{eqn}) holds. We obtain in particular that
$$
u_t(0)=\Phi_1(0)(u_0, u_1)-\frac 1 {2r}\int_0^r \rho|u_0(\rho)|^Nu_0(\rho) \dd \rho=u_1-\frac 1 {2r}\int_0^r \rho|u_0(\rho)|^Nu_0(\rho) \dd \rho.
$$
Multiplying (\ref{eqn}) by $r$ we obtain
$$
(ru)_t+(ru)_r+\frac 1 2\int_0^r \rho |u(\rho)|^Nu(\rho) \dd \rho = (r\Phi_0(t)(u_0, u_1))_r + r\Phi_1(t)(u_0, u_1).
$$
Therefore, taking a $t$ derivative and using the same relation again,
$$\begin{aligned}
(ru)_{tt} &= -(ru)_{rt} - \frac {N+1} 2 \int_0^r \rho |u(\rho)|^N u_t(\rho) \dd \rho  + (r\Phi_1(t)(u_0, u_1))_r + \\
&+ (r\Phi_1(t)(u_0, u_1))_t\\
&= (ru)_{rr} + \frac 1 2 r |u|^N u - (r\Phi_0(t)(u_0, u_1))_{rr} - (r\Phi_1(t)(u_0, u_1))_r - \\
&\frac {N+1} 2 \int_0^r \rho |u(\rho)|^N u_t(\rho) \dd \rho + (r\Phi_1(t)(u_0, u_1))_r+(r\Phi_1(t)(u_0, u_1))_t
\end{aligned}$$
or in other words, since $(r\Phi_0(t)(u_0, u_1))_{rr} = (r\Phi_1(t)(u_0, u_1))_t$,
$$
u_{tt} = u_{rr} + \frac 2 r u_r + \frac 1 2 |u|^N u - \frac {N+1} {2r} \int_0^r \rho |u(\rho)|^N u_t(\rho) \dd \rho,
$$
that is (\ref{equivalent}), which is equivalent to (\ref{eq_outgoing}).
\end{proof}

\section{Existence results}\lb{local_ex}


Our first result is that the equation (\ref{eq_outgoing}) is well-posed in the critical sense for $(u_0, u_1) \in \dot H^{s_c} \times \dot H^{s_c-1}$ initial data. In addition, if the initial data is outgoing, then the solution lives inside a thickened cone, which we can identify as a sharp Huygens principle.

\begin{proposition}\lb{well-posed} Assume that $N \geq 4$ and $(u_0, u_1) \in \dot H^{s_c} \times \dot H^{s_c-1}$ are radial. Then there exist $T>0$ and a corresponding solution $u$ to (\ref{eq_out}) on $I=[0, T]$ such that
$$
\|u\|_{L^\infty_t \dot H^{s_c}_x(\R^3 \times I) \cap L^{N/2}_t L^\infty_x(\R^3 \times I)} \les \|u_0\|_{\dot H^{s_c}} + \|u_1\|_{\dot H^{s_c-1}}.
$$
If $(u_0, u_1)$ are also small then $I=[0, \infty)$.

Finally, if $u_0$ and $u_1$ are outgoing and supported on $ \ov{\R^3 \setminus B(0, R)}$ with $R \geq 0$, then for $t \in I$ $u(t)$ and $u_t(t)$ are supported on $\ov{\R^3 \setminus B(0, R+t)}$.
\end{proposition}
\begin{proof}[Proof of Proposition \ref{well-posed}]
We work with the (\ref{eq_out}) version of the equation. We use a fixed point argument in the space $L^\infty_t \dot H^{s_c}_x(\R^3 \times I) \cap L^{N/2}_t L^\infty_x(\R^3 \times I)$. We put the inhomogenous term $|u|^N u$ in $L^1_t \dot H^{s_c-1}_x(\R^3 \times I)$ and $\frac 1 r \int_0^r \rho |u(\rho, t)|^N u(\rho, t) \dd \rho$ into $L^1_t \dot H^{s_c}_x(\R^3 \times I)$. Indeed,
$$\begin{aligned}
\||u|^Nu\|_{L^1_t \dot H^{s_c-1}_x(\R^3 \times I)} &\les \|u\|^{N/2}_{L^{N/2}_t L^\infty_x(\R^3 \times I)} \|u\|^{N/2}_{L^\infty_t L^{3N/2}_x(\R^3 \times I)} \|u\|_{\dot L^\infty_t W^{s_c-1, 6}_x(\R^3 \times I)} \\
&\les \|u\|^{N/2}_{L^{N/2}_t L^\infty_x(\R^3 \times I)} \|u\|^{N/2+1}_{L^\infty_t \dot H^{s_c}_x(\R^3 \times I)}.
\end{aligned}$$
The conclusion follows as in the case of the usual supercritical equation (\ref{defocusing}), since by Lemma \ref{bounds}
$$
\Big\|\frac 1 r \int_0^r \rho |u(\rho, t)|^N u(\rho, t) \dd \rho\Big\|_{\dot H^{s_c}_x} \les \||u(\rho, t)|^N u(\rho, t)\|_{\dot H^{s_c-1}_x}
$$
and consequently
$$
\Big\|\frac 1 r \int_0^r \rho |u(\rho, t)|^N u(\rho, t) \dd \rho\Big\|_{L^\infty_t \dot H^{s_c}_x(\R^3 \times I)} \les \|u\|^{N/2}_{L^{N/2}_t L^\infty_x(\R^3 \times I)} \|u\|^{N/2+1}_{L^\infty_t \dot H^{s_c}_x(\R^3 \times I)}.
$$
The interval $I=[0, T]$ is chosen such that the linear evolution of the initial data has small norm on $I$, $\|\Phi_0(t)(u_0, u_1)\|_{L^{N/2}_t L^\infty_x(\R^3 \times I)} << 1$.

The last conclusion concerning support follows because it holds for each iteration in the fixed point scheme, due to the outgoing projection in the equation.
\end{proof}

%
%
%
%

Next, we prove an existence result in the subcritical sense (where the time of existence depends only on the norm of the initial data) for equation (\ref{eq_outgoing}). In addition, we also need to show that the solution depends continuously on the initial data and that regularity and some decay are preserved.
\begin{proposition}\lb{local_existence}
Assume that $N \geq 2$ and $(u_0, u_1) \in ((\dot H^1 \cap L^\infty) \times L^2)_{out}$ are radial and outgoing. Then there exist an interval $I=[0, T]$,
$$
T\geq C \min(\|u_0\|_{\dot H^1}^{-2} \|u_0\|_{L^\infty}^{-N+2}, \|u_0\|_{L^\infty}^{-N}),$$
and a corresponding solution $u$ to (\ref{eq_outgoing}) defined on $\R^3 \times I$
such that
$$
\|u\|_{L^\infty_t \dot H^1_x(\R^3 \times I)} \les \|u_0\|_{\dot H^1},\ \|u\|_{L^\infty_{t, x}(\R^3 \times I)} \les \|u_0\|_{L^\infty}.
$$
The solution $u$ depends continuously on the initial data, i.e.\ every $u_0 \in \dot H^1 \cap L^\infty$ has a neighborhood $\mc N=\{\tilde u_0 \mid \|\tilde u_0\|_{\dot H^1} \leq 2 \|u_0\|_{\dot H^1},\ \|\tilde u_0\|_{L^\infty} \leq 2 \|u_0\|_{L^\infty}\}$ such that if $\tilde u_0 \in \mc N$ then the corresponding solution $\tilde u$ is also defined on $\R^3 \times I$ and
\be\lb{cont_dep}
\|\tilde u - u\|_{L^\infty_t \dot H^1_x(\R^3 \times I) \cap L^\infty_{t, x}(\R^3 \times I)} \les \|\tilde u_0 - u_0\|_{\dot H^1 \cap L^\infty}.
\ee

If in addition $(u_0, u_1) \in \dot H^2 \times \dot H^1$, then $u \in L^\infty_t \dot H^2_x(\R^3 \times I)$; furthermore, for $v$ given by (\ref{eq_outgoing'}), $v \in L^\infty_t (\dot H^1_x \cap L^2_x)(\R^3 \times I)$. Likewise, if $u_0 \in \langle x \rangle^{-1} L^\infty$, then $u \in \langle x \rangle^{-1} L^\infty_{t, x}(\R^3 \times I)$.

If $N \geq 4$ and the initial data are also sufficiently small, i.e.
\be\lb{small_data}
\|u_0\|_{\dot H^1}^4 \|u_0\|_{L^\infty}^{N-4} + \|u_0\|_{\dot H^1}^2 \|u_0\|_{L^\infty}^{N-2} << 1,
\ee
then the corresponding solution $u$ to (\ref{eq_outgoing}) exists globally, forward in time, and
$$
\|u\|_{L^\infty_t \dot H^1_x \cap L^2_t L^\infty_x} \les \|u_0\|_{\dot H^1},\ \|u\|_{L^\infty_{t, x}} \les \|u_0\|_{L^\infty}.
$$
These solutions also depend continuously on initial data and preserve regularity and decay for all time.

\end{proposition}
Due to the outgoing nature of the equation, if the initial data are supported on $\ov{\R^3 \setminus B(0, R)}$ for some $R \geq 0$, then $u(t)$ is supported on $\ov{\R^3 \setminus B(0, R+t)}$ for $t \in I$. This is true for the solution $u$ because it is true for all the iterates in the fixed point argument.
\begin{proof}[Proof of Proposition \ref{local_existence}] We work with the (\ref{eq_out}) version of the equation. Even though some alternative results are possible using (\ref{equivalent}), this other version destroys the outgoing structure of the equation, leading to worse bounds.
	
We apply a fixed point argument in the space $L^\infty_t \dot H^1_x(\R^3 \times I) \cap L^\infty_{t, x}(\R^3 \times I)$. Consider the linearized version of equation (\ref{eq_out})
\be\lb{eq_linearized}\begin{aligned}
&u(t) = \cos(t\sqrt{-\Delta}) u_0 + \frac {\sin(t\sqrt{-\Delta})}{\sqrt{-\Delta}} u_1 - \frac 1 2 \int_0^t \cos((t-s)\sqrt{-\Delta}) \\
&\Big(\frac 1 r \int_0^r \rho |\tilde u(\rho, s)|^N \tilde u(\rho, s) \dd \rho\Big) \dd s + \frac 1 2 \int_0^t \frac {\sin((t-s)\sqrt{-\Delta})}{\sqrt{-\Delta}} |\tilde u(s)|^N\tilde u(s) \dd s.
\end{aligned}\ee
Note that both the initial data and the inhomogenous term are outgoing, so the whole solution is outgoing.

Assume that $N \geq 2$. In $\dot H^1$ we see that, since $\dot H^1 \subset L^{6, 2}$,
$$\begin{aligned}
\|u\|_{L^\infty_t \dot H^1_x(\R^3 \times I)} &\les \|u_0\|_{\dot H^1} + \|u_1\|_{L^2} + T \Big(\Big\|\frac 1 r \int_0^r \rho |\tilde u(\rho, s)|^N \tilde u(\rho, s) \dd \rho\Big\|_{L^\infty_s \dot H^1_x(\R^3 \times I)} + \\
&+ \||\tilde u(s)|^N \tilde u(s)\|_{L^\infty_s L^2_x(\R^3 \times I)}\Big) \\
&\les \|u_0\|_{\dot H^1} + T \||\tilde u|^N \tilde u\|_{L^\infty_t L^2_x(\R^3 \times I)} \\
&\les \|u_0\|_{\dot H^1} + T \|\tilde u\|_{L^\infty_t \dot H^1_x(\R^3 \times I)}^3 \|\tilde u\|_{L^\infty_{t, x}(\R^3 \times I)}^{N-2}.
\end{aligned}$$
Here we have used Lemma \ref{bounds}. In $L^\infty$, in view of (\ref{linf}) and of the outgoing character of both the initial data and the inhomogenous term,
$$\begin{aligned}
\|u\|_{L^\infty_{t,x}(\R^3 \times I)} &\les \|u_0\|_{L^\infty} + T\Big\|\frac 1 r \int_0^r \rho |\tilde u(\rho, s)|^N \tilde u(\rho, s) \dd \rho\Big\|_{L^\infty_{s, x}(\R^3 \times I)} \\
&\les \|u_0\|_{L^\infty} + T \||\tilde u(s)|^N \tilde u(s)\|_{L^\infty_s L^{3, 1}_x(\R^3 \times I) \cap L^\infty_{s, x}(\R^3 \times I)} \\
&\les \|u_0\|_{L^\infty} + T (\|\tilde u\|^2_{L^\infty_t \dot H^1_x(\R^3 \times I)} \|\tilde u\|^{N-1}_{L^\infty_{t, x}(\R^3 \times I)} + \|\tilde u\|^{N+1}_{L^\infty_{t, x}(\R^3 \times I)}).
\end{aligned}$$
Here we have also used Lemma \ref{improved}. Thus, assuming that
\be\lb{condition}
\|\tilde u\|_{L^\infty_t \dot H^1_x(\R^3 \times I)} \les \|u_0\|_{\dot H^1},\ \|\tilde u\|_{L^\infty_{t, x}(\R^3 \times I)} \les \|u_0\|_{L^\infty},
\ee
and that
\be\lb{cond}
T \leq c \min(\|u_0\|_{\dot H^1}^{-2} \|u_0\|_{L^\infty}^{-N+2}, \|u_0\|_{L^\infty}^{-N})
\ee
with $c$ sufficiently small, we retrieve the same conclusion (\ref{condition}) for $u$.

Next, we show that the mapping $\tilde u \mapsto u$ is a contraction. Indeed, in the same way as above one can prove that for two pairs $u^1$ and $\tilde u^1$, respectively $u^2$ and $\tilde u^2$ that both solve (\ref{eq_linearized}),
\be\lb{45}\begin{aligned}
\|u^1-u^2\|_{L^\infty_t \dot H^1_x(\R^3 \times I)} \les T \|\tilde u^1-\tilde u^2\|_{L^\infty_t \dot H^1_x(\R^3 \times I)} (\|\tilde u^1\|_{L^\infty_t \dot H^1_x(\R^3 \times I)}^2 \\
\|\tilde u^1\|_{L^\infty_{t, x}(\R^3 \times I)}^{N-2} + \|\tilde u^2\|_{L^\infty_t \dot H^1_x(\R^3 \times I)}^2 \|\tilde u^2\|_{L^\infty_{t, x}(\R^3 \times I)}^{N-2})
\end{aligned}\ee
and
\be\lb{46}\begin{aligned}
\|u^1 - u^2\|_{L^\infty_{t, x}(\R^3 \times I)} \les T \|\tilde u^1 - \tilde u^2\|_{L^\infty_{t, x}(\R^3 \times I)} (\|\tilde u^1\|^2_{L^\infty_t \dot H^1_x(\R^3 \times I)} \|\tilde u^1\|^{N-2}_{L^\infty_{t, x}(\R^3 \times I)} + \\
+ \|\tilde u^1\|^N_{L^\infty_{t, x}(\R^3 \times I)} + \|\tilde u^2\|^2_{L^\infty_t \dot H^1_x(\R^3 \times I)} \|\tilde u^2\|^{N-2}_{L^\infty_{t, x}(\R^3 \times I)} + \|\tilde u^2\|^N_{L^\infty_{t, x}(\R^3 \times I)}).
\end{aligned}\ee
Therefore, again assuming condition (\ref{cond}) with $c$ sufficiently small, it follows that the mapping is indeed a contraction on the set of $u$ satisfying condition (\ref{condition}) and it has a fixed point $u=\tilde u$. This fixed point is a solution of (\ref{eq_out}) on $\R^3 \times I$ with the desired properties.

In order to prove the continuous dependence on the initial data, first note that if $\|\tilde u_0\|_{\dot H^1} \leq 2\|u_0\|_{\dot H^1}$ and $\|\tilde u_0\|_{L^\infty} \leq 2\|u_0\|_{L^\infty}$, then the solution $\tilde u$ also exists on an interval $[0, T]$ where
$$
T = c \min(\|u_0\|_{\dot H^1}^{-2} \|u_0\|_{L^\infty}^{-N+2}, \|u_0\|_{L^\infty}^{-N}),
$$
for perhaps a smaller value of $c$. After making this adjustment, take the two versions of equation (\ref{eq_outgoing}) satisfied by $u$ and $\tilde u$ and subtract them from one another. We obtain, analogously to (\ref{45}) and (\ref{46}), that
$$\begin{aligned}
\|\tilde u - u\|_{L^\infty_t \dot H^1_x(\R^3 \times I) \cap L^\infty_{t, x}(\R^3 \times I)} \les \|\tilde u_0-u_0\|_{\dot H^1 \cap L^\infty} + T \|\tilde u - u\|_{L^\infty_t \dot H^1_x(\R^3 \times I) \cap L^\infty_{t, x}(\R^3 \times I)} \\
(\|u_0\|_{\dot H^1}^{2} \|u_0\|_{L^\infty}^{N-2} + \|u_0\|_{L^\infty}^N + \|\tilde u_0\|_{\dot H^1}^{2} \|\tilde u_0\|_{L^\infty}^{N-2} + \|\tilde u_0\|_{L^\infty}^N).
\end{aligned}$$
We obtain (\ref{cont_dep}) for sufficiently small $T$ as in (\ref{cond}).

Next, if the initial data are in $\dot H^2 \times \dot H^1$ we obtain the preservation of regularity for free:
$$\begin{aligned}
\|u\|_{L^\infty_t \dot H^2_x(\R^3 \times I)} &\les \|u_0\|_{\dot H^2} + \|u_1\|_{\dot H^1} + T \Big(\Big\|\frac 1 r \int_0^r \rho|u(\rho, s)|^N u(\rho, s) \dd \rho\Big\|_{L^\infty_s \dot H^2_x(\R^3 \times I)} + \\
&+ \||u(s)|^Nu(s)\|_{L^\infty_s \dot H^1_x(\R^3 \times I)}\Big) \\
& \les \|u_0\|_{\dot H^2} + \|u_1\|_{\dot H^1} + T \||u(s)|^Nu(s)\|_{L^\infty_s \dot H^1_x(\R^3 \times I)} \\
&\les \|u_0\|_{\dot H^2} + \|u_1\|_{\dot H^1} + T \|u\|_{L^\infty_t \dot H^1_x(\R^3 \times I)} \|u\|_{L^\infty_{t, x}(\R^3 \times I)}^N < \infty.
\end{aligned}$$
See Lemma \ref{bounds} for the bound used here.

Following the (\ref{eq_outgoing'}) formulation of equation (\ref{eq_outgoing}), let
$$
v=u_t+\frac 1 {2r} \int_0^r \rho|u_0(\rho)|^N u_0(\rho) \dd \rho.
$$
Then, by (\ref{eq_ut}),
$$
v=\Phi_1(u_0, u_1)+\int_0^t \Phi_1(t-s) P_+(0, |u(s)|^Nu(s)) \dd s.
$$
It immediately follows, same as above, that
$$
\|v\|_{L^\infty_t L^2_x(\R^3 \times I)} \les \|u_0\|_{\dot H^1} + T \|u\|^3_{L^\infty_t \dot H^1(\R^3 \times I)} \|u\|^{N-2}_{L^\infty_{t, x}(\R^3 \times I)}
$$
and
$$
\|v\|_{L^\infty_t \dot H^1_x(\R^3 \times I)} \les \|u_0\|_{\dot H^2} + \|u_1\|_{\dot H^1} + T \|u\|_{L^\infty_t \dot H^1_x(\R^3 \times I)} \|u\|_{L^\infty_{t, x}(\R^3 \times I)}^N.
$$

Concerning the preservation of decay, a rigorous way to prove it is to include it in the fixed point argument. Indeed, consider a pair $u$ and $\tilde u$ that together solve (\ref{eq_linearized}) for initial data $u_0 \in |x|^{-1} L^\infty$. By (\ref{linfw})
$$\begin{aligned}
\|u\|_{|x|^{-1} L^\infty_{t, x}(\R^3 \times I)} &\les \|u_0\|_{|x|^{-1} L^\infty_x} + T \Big\|\frac 1 r \int_0^r \rho|\tilde u(\rho, s)|^N \tilde u(\rho, s) \dd \rho\Big\|_{|x|^{-1} L^\infty_{s, x} (\R^3 \times I)} \\
&\les \|u_0\|_{|x|^{-1} L^\infty_x} + T \Big\|\frac {|\tilde u(s)|^N \tilde u(s)}{|x|}\Big\|_{L^\infty_s L^1_x(\R^3 \times I)} \\
&\les \|u_0\|_{|x|^{-1} L^\infty_x} + T \|\tilde u\|_{|x|^{-1}L^\infty_{t, x}(\R^3 \times I)} \|u\|_{L^\infty_t \dot H^1_x}^2 \|u\|_{L^\infty_{t, x}}^{N-2}.
\end{aligned}$$
Here we have also used the fact that $\dot H^1 \subset L^{6, 2}$. In the same manner we obtain that
$$\begin{aligned}
\|u^1-u^2\|_{|x|^{-1}L^\infty_{t, x}(\R^3 \times I)} \les T \|\tilde u^1 - \tilde u^2\|_{|x|^{-1}L^\infty_{t, x}(\R^3 \times I)} (\|\tilde u^1\|_{L^\infty_t \dot H^1_x}^2 \|\tilde u^1\|_{L^\infty_{t, x}}^{N-2} + \\
+ \|\tilde u^2\|_{L^\infty_t \dot H^1_x}^2 \|\tilde u^2\|_{L^\infty_{t, x}}^{N-2}).
\end{aligned}$$
Therefore, in this case the fixed point of the contraction mapping will be a solution $u \in |x|^{-1} L^\infty_{t, x}(\R^3 \times I)$.

Next, assume that $N \geq 4$ and that the initial data are small. We use a fixed point argument in the space $L^\infty_t \dot H^1_x \cap L^2_t L^\infty_x \cap L^\infty_{t, x}$. Note that
$$\begin{aligned}
\|u\|_{L^\infty_t \dot H^1_x \cap L^2_t L^\infty_x} &\les \|u_0\|_{\dot H^1} + \|u_1\|_{L^2} + \Big\|\frac 1 r \int_0^r \rho |\tilde u(\rho, s)|^N \tilde u(\rho, s) \dd \rho\Big\|_{L^1_s \dot H^1_x} + \\
&+ \||\tilde u(s)|^N \tilde u(s)\|_{L^1_s L^2_x} \\
&\les \|u_0\|_{\dot H^1} + \||\tilde u(s)|^N \tilde u(s)\|_{L^1_s L^2_x} \\
&\les \|u_0\|_{\dot H^1} + \|\tilde u\|^3_{L^\infty_t \dot H^1_x} \|\tilde u\|^2_{L^2_t L^\infty_x} \|\tilde u\|_{L^\infty_{t, x}}^{N-4}
\end{aligned}$$
and, in view of (\ref{linf}),
$$\begin{aligned}
\|u\|_{L^\infty_{t, x}} &\les \|u_0\|_{L^\infty} + \Big\|\frac 1 r \int_0^r \rho |\tilde u(\rho, s)|^N \tilde u(\rho, s) \dd \rho\Big\|_{L^1_s L^\infty_x} \\
&\les \|u_0\|_{L^\infty} + \||\tilde u(s)|^N \tilde u(s)\|_{L^1_s L^{3, 1}_x \cap L^1_s L^\infty_x} \\
&\les \|u_0\|_{L^\infty} + \|\tilde u\|_{L^\infty_t \dot H^1_x}^2 \|\tilde u\|_{L^2_t L^\infty_x}^2 \|\tilde u\|_{L^\infty_{t, x}}^{N-3} + \|\tilde u\|_{L^2_t L^\infty_x}^2 \|\tilde u\|_{L^\infty_{t, x}}^{N-1}.
\end{aligned}$$
Therefore, assuming that
\be\lb{assum}
\|\tilde u\|_{L^\infty_t \dot H^1_x \cap L^2_t L^\infty_x} \les \|u_0\|_{\dot H^1},\ \|\tilde u\|_{L^\infty_{t, x}} \les \|u_0\|_{L^\infty},
\ee
and
\be\lb{assump}
\|u_0\|_{\dot H^1}^4 \|u_0\|_{L^\infty}^{N-4} + \|u_0\|_{\dot H^1}^2 \|u_0\|_{L^\infty}^{N-2} << 1,
\ee
we retrieve the same conclusion (\ref{assum}) for $u$. We also prove that the mapping $\tilde u \mapsto u$ is a contraction. Indeed, for two pairs $u_1$ and $\tilde u_1$, respectively $u_2$ and $\tilde u_2$ that both solve (\ref{eq_linearized}),
$$\begin{aligned}
\|u_1-u_2\|_{L^\infty_t \dot H^1_x \cap L^2_t L^\infty_x} &\les \|\tilde u_1-\tilde u_2\|_{L^\infty_t \dot H^1_x} (\|\tilde u_1\|^2_{L^\infty_t \dot H^1_x} \|\tilde u_1\|^2_{L^2_t L^\infty_x} \|\tilde u_1\|_{L^\infty_{t, x}}^{N-4} + \\
&+ \|\tilde u_2\|^2_{L^\infty_t \dot H^1_x} \|\tilde u_2\|^2_{L^2_t L^\infty_x} \|\tilde u_2\|_{L^\infty_{t, x}}^{N-4})
\end{aligned}$$
and
$$\begin{aligned}
\|u_1-u_2\|_{L^\infty_{t, x}} \les \|\tilde u_1-\tilde u_2\|_{L^\infty_{t, x}} (\|\tilde u_1\|_{L^\infty_t \dot H^1_x}^2 \|\tilde u_1\|_{L^2_t L^\infty_x}^2 \|\tilde u_1\|_{L^\infty_{t, x}}^{N-4} + \|\tilde u_1\|_{L^2_t L^\infty_x}^2 \|\tilde u_1\|_{L^\infty_{t, x}}^{N-2} + \\
+ \|\tilde u_2\|_{L^\infty_t \dot H^1_x}^2 \|\tilde u_2\|_{L^2_t L^\infty_x}^2 \|\tilde u_2\|_{L^\infty_{t, x}}^{N-4} + \|\tilde u_2\|_{L^2_t L^\infty_x}^2 \|\tilde u_2\|_{L^\infty_{t, x}}^{N-2}).
\end{aligned}$$
Thus the mapping $\tilde u \mapsto u$ is a contraction under condition (\ref{assump}) on the set of $u$ that fulfill (\ref{assum}). Its fixed point is a solution of (\ref{eq_out}) with the desired properties.

The continuous dependence on initial data and the persistence of regularity and of decay are proved in exactly the same way as in the large data case.
\end{proof}

We next prove a local existence result in the $L^{N+2}$ norm for $N > 2$, under the assumption that the initial data are supported away from zero. We also find a corresponding small data global existence result for $N>4$.

Besides these existence results, we are interested in proving the continuous dependence of solutions on initial data and the persistence of regularity and decay.

\begin{proposition}\lb{LNexistence} Assume that $N>2$ and $(u_0, u_1)$ are radial and outgoing initial data with $u_0 \in L^{N+2}$ and $\supp u_0 \subset \ov{\R^3 \setminus B(0, R)}$ for some $R>0$. Then there exists a corresponding solution $u$ to (\ref{eq_outgoing}) defined on $\R^3\times I$, where $I=[0, T]$ with $T \geq C R^{\frac{2N-2}{N+2}} \|u_0\|_{L^{N+2}}^{-N}$, and
$$
\|u\|_{L^\infty_t L^{N+2}_x(\R^3 \times I)} \les \|u_0\|_{L^{N+2}}.
$$
The solution $u$ depends continuously on the initial data: every $u_0$ as above has a neighborhood $\mc N=\{\tilde u_0 \mid \|\tilde u_0\| \leq 2 \|u_0\|_{L^{N+2}},\ \supp u_0 \subset \ov{\R^3 \setminus B(0, R)}\}$ such that if $\tilde u_0 \in \mc N$ then the corresponding solution $\tilde u$ is also defined on $\R^3 \times I$ and
$$
\|\tilde u - u\|_{L^\infty_t L^{N+2}_x(\R^3 \times I)} \les \|\tilde u_0 - u_0\|_{L^{N+2}}.
$$

In addition, if $u_0 \in \dot H^1 \cap L^\infty$ then $u \in L^\infty_t \dot H^1_x(\R^3 \times I) \cap L^\infty_{t, x}(\R^3 \times I)$ and if $u_0 \in \langle x \rangle^{-1} L^\infty$ then $u \in \langle x \rangle^{-1} L^\infty_{t, x}(\R^3 \times I)$.

Assume $N>4$. If the initial data
\be\lb{LNsmall}
R^{\frac{4-N}{N+2}} \|u_0\|_{L^{N+2}}^N << 1
\ee
are sufficiently small, then there exists a corresponding global solution $u$ to (\ref{eq_outgoing}) on $\R^3 \times [0, \infty)$ such that $\|u\|_{L^\infty_t L^{N+2}_x} \les \|u_0\|_{L^{N+2}}$.

In addition, the solution depends continuously on the initial data; if $u_0 \in L^\infty$ then $u \in L^\infty_{t, x}$; and there is persistence of regularity and of decay for all time $t \geq 0$.
\end{proposition}
By standard arguments (i.e.\ dominated convergence, see \cite{becsof}), one can show that in fact when $N>4$ $\|u(t)\|_{L^{N+2}} \to 0$ as $t \to \infty$.

Again, due to the outgoing nature of the equation, if the initial data are supported on $\ov{\R^3 \setminus B(0, R)}$ for some $R \geq 0$, then $u(t)$ is supported on $\ov{\R^3 \setminus B(0, R+t)}$ for $t \in I$. This is true for the solution $u$ because it is true for all the iterates in the fixed point argument.
\begin{proof}[Proof of Proposition \ref{LNexistence}] The proof is based on a fixed point argument. Consider the linearized version (\ref{eq_linearized}) of equation (\ref{eq_out})
$$\begin{aligned}
&u(t) = \cos(t\sqrt{-\Delta}) u_0 + \frac {\sin(t\sqrt{-\Delta})}{\sqrt{-\Delta}} u_1 - \frac 1 2 \int_0^t \cos((t-s)\sqrt{-\Delta}) \\
&\Big(\frac 1 r \int_0^r \rho |\tilde u(\rho, s)|^N \tilde u(\rho, s) \dd \rho\Big) \dd s + \frac 1 2 \int_0^t \frac {\sin((t-s)\sqrt{-\Delta})}{\sqrt{-\Delta}} |\tilde u(s)|^N\tilde u(s) \dd s.
\end{aligned}$$
Due to Proposition \ref{est_outgoing}, it immediately follows that
$$
\|u\|_{L^\infty_t L^{N+2}_x (\R^3 \times I)} \les \|u_0\|_{L^{N+2}} + T \Big\|\frac 1 r \int_0^r \rho |\tilde u(\rho, s)|^N \tilde u(\rho, s) \dd \rho\Big\|_{L^\infty_s L^{N+2}_x (\R^3 \times I)}.
$$

Assume that $\tilde u$ (and hence $u$) is supported on $\ov{\R^3 \setminus B(0, R)}$ as well for all times $t \in I$. By Lemma \ref{awayfromzero}, we have that when $N+2 > \frac {3(N+2)}{N+1}$, i.e. when $N>2$,
$$\begin{aligned}
\Big\|\frac 1 r \int_0^r \rho |\tilde u(\rho, s)|^N \tilde u(\rho, s) \dd \rho\Big\|_{L^{N+2}_x} &\les R^{1-3\frac{N+1}{N+2}+\frac 3 {N+2}} \||\tilde u(s)|^N \tilde u(s)\|_{L^{\frac {N+2}{N+1}}} \\
&= R^{\frac{2-2N}{N+2}} \|\tilde u(s)\|_{L^{N+2}}^{N+1}.
\end{aligned}$$

Consequently
$$
\|u\|_{L^\infty_t L^{N+2}_x(\R^3 \times I)} \les \|u_0\|_{L^{N+2}} + T R^{\frac{2-2N}{N+2}} \|\tilde u\|_{L^\infty_t L^{N+2}_x (\R^3 \times I)}^{N+1}.
$$
If $\|\tilde u\|_{L^\infty_t L^{N+2}_x (\R^3 \times I)} \les \|u_0\|_{L^{N+2}}$ and $T \leq c R^{\frac{2N-2}{N+2}} \|u_0\|_{L^{N+2}}^{-N}$ with $c$ sufficiently small, then we retrieve the same conclusion for $u$: $\|u\|_{L^\infty_t L^{N+2}_x (\R^3 \times I)} \les \|u_0\|_{L^{N+2}}$.

In a similar manner one can prove that for two pairs $u^1$ and $\tilde u^1$, respectively $u^2$ and $\tilde u^2$, both satisfying (\ref{eq_linearized}),
$$\begin{aligned}
\|u^1-u^2\|_{L^\infty_t L^{N+2}_x(\R^3 \times I)} &\les T R^{\frac{2-2N}{N+2}} \|\tilde u^1-\tilde u^2\|_{L^\infty_t L^{N+2}_x (\R^3 \times I)} (\|\tilde u^1\|_{L^\infty_t L^{N+2}_x (\R^3 \times I)}^N + \\
&+ \|\tilde u^2\|_{L^\infty_t L^{N+2}_x (\R^3 \times I)}^N).
\end{aligned}$$
Thus the mapping $\tilde u \mapsto u$ is a contraction under the same condition on $T$, $T \leq c R^{\frac{2N-2}{N+2}} \|u_0\|_{L^{N+2}}^{-N}$, on the ball $\{u \mid \|u\|_{L^\infty_t L^{N+2}_x (\R^3 \times I)} \les \|u_0\|_{L^{N+2}}\}$. The fixed point $u$ is a solution of (\ref{eq_outgoing}) with the desired properties.

In the same manner we can show that, if $u$ and $\tilde u$ are two solutions with initial data $u_0$ and $\tilde u_0$, then
$$\begin{aligned}
\|\tilde u-u\|_{L^\infty_t L^{N+2}_x(\R^3 \times I)} \les &\|\tilde u_0-u_0\|_{L^{N+2}} + T R^{\frac{2-2N}{N+2}} \|\tilde u^1-\tilde u^2\|_{L^\infty_t L^{N+2}_x (\R^3 \times I)} \\
&(\|\tilde u^1\|_{L^\infty_t L^{N+2}_x (\R^3 \times I)}^N + \|\tilde u^2\|_{L^\infty_t L^{N+2}_x (\R^3 \times I)}^N).
\end{aligned}$$
This proves the continuous dependence of the solution on the initial data.

Next, assume that the initial data $u_0 \in L^\infty$. Then, by (\ref{linf}) and Lemma \ref{awayfromzero}, since $\frac {N+2}{N+1} \leq 2$,
$$\begin{aligned}
\|u\|_{L^\infty_{t, x}(\R^3 \times I)} &\les \|u_0\|_{L^\infty} + T \Big\|\frac 1 r \int_0^r \rho |u(\rho, s)|^N u(\rho, s) \dd \rho\Big\|_{L^\infty_{s, x}(\R^3 \times I)} \\
&\les \|u_0\|_{L^\infty} + T R^{1-3\frac {N+1}{N+2}} \||u(s)|^Nu(s)\|_{L^\infty_s L^{\frac {N+2}{N+1}}_x (\R^3 \times I)} \\
&\les \|u_0\|_{L^\infty} + T R^{-\frac{2N+1}{N+2}} \|u\|^{N+1}_{L^\infty_s L^{N+2}_x(\R^3 \times I)} \\
&\les \|u_0\|_{L^\infty} + R^{-\frac 3 {N+2}} \|u_0\|_{L^{N+2}} < \infty.
\end{aligned}$$
Assume in addition that $u_0 \in \dot H^1$. Then, since $2(N+1) \geq N+2$,
$$\begin{aligned}
\|u\|_{L^\infty_t \dot H^1_x(\R^3 \times I)} &\les \|u_0\|_{\dot H^1} + \|u_1\|_{L^2} + T\Big(\Big\|\frac 1 r \int_0^r \rho |u(\rho, s)|^N u(\rho, s) \dd \rho\Big\|_{L^\infty_s \dot H^1_x(\R^3 \times I)} + \\
&+ \||u(s)|^Nu(s)\|_{L^\infty_s L^2_x(\R^3 \times I)}\Big) \\
&\les \|u_0\|_{\dot H^1} + \||u(s)|^Nu(s)\|_{L^\infty_s L^2_x(\R^3 \times I)} \\
&\les \|u_0\|_{\dot H^1} + \|u\|^{N+1}_{L^\infty_t L^{2(N+1)}_x(\R^3 \times I)} < \infty.
\end{aligned}$$
Finally, assume that $u_0 \in \langle x \rangle^{-1} L^\infty$. Then we already know that $u \in L^\infty_{t, x}(\R^3 \times I)$ and by (\ref{linfw}), since for $N>1$ $\frac {N+2}{N+1}<\frac 3 2$,
$$\begin{aligned}
\|u\|_{|x|^{-1}L^\infty_{t, x}} &\les \|u_0\|_{|x|^{-1}L^\infty} + T \Big\|\frac 1 r \int_0^r \rho |u(\rho, s)|^N u(\rho, s) \dd \rho\Big\|_{|x|^{-1} L^\infty_{t, x}(\R^3 \times I)} \\
&\les \|u_0\|_{|x|^{-1}L^\infty} + T\Big\|\frac {|u(s)|^Nu(s)}{|x|}\Big\|_{L^\infty_s L^1_x(\R^3 \times I)} \\
&\les \|u_0\|_{|x|^{-1}L^\infty} + \|u\|_{L^\infty_t L^{N+2}_x(\R^3 \times I) \cap L^\infty_{t, x}(\R^3 \times I)}^{N+1}.
\end{aligned}$$

In the small initial data case, we similarly see that
$$
\|u\|_{L^\infty_t L^{N+2}_x} \les \|u_0\|_{L^{N+2}} + \Big\|\frac 1 r \int_0^r \rho |\tilde u(\rho, s)|^N \tilde u(\rho, s) \dd \rho\Big\|_{L^1_s L^{N+2}_x}.
$$
We then take advantage of the fact that, due to the outgoing nature of the equation, we may assume that $\supp \tilde u(s) \subset \ov{\R^3 \setminus B(0, R+s)}$ (and hence same for $u$). Therefore we obtain as above that
$$
\Big\|\frac 1 r \int_0^r \rho |\tilde u(\rho, s)|^N \tilde u(\rho, s) \dd \rho\Big\|_{L^{N+2}_x} \les (R+s)^{\frac{2-2N}{N+2}} \|\tilde u(s)\|_{L^{N+2}}^{N+1}.
$$
For $\frac {2-2N}{N+2}<-1$, i.e. $N>4$, $\int_0^\infty (R+s)^{\frac{2-2N}{N+2}} \dd s = C R^{\frac {4-N}{N+2}} < \infty$. Thus
$$
\Big\|\frac 1 r \int_0^r \rho |\tilde u(\rho, s)|^N \tilde u(\rho, s) \dd \rho\Big\|_{L^1_s L^{N+2}_x} \les R^{\frac{4-N}{N+2}} \|\tilde u\|_{L^\infty_t L^{N+2}_x}^{N+1}.
$$
In conclusion
$$
\|u\|_{L^\infty_t L^{N+2}_x} \les \|u_0\|_{L^{N+2}} + R^{\frac{4-N}{N+2}} \|\tilde u\|_{L^\infty_t L^{N+2}_x}^{N+1}.
$$
Thus, assuming that $\|\tilde u\|_{L^\infty_t L^{N+2}_x} \les \|u_0\|_{L^{N+2}}$ and that $R^{\frac {4-N}{N+2}} \|u_0\|_{L^{N+2}}^N << 1$, we obtain the same conclusion for $u$. In a similar manner one shows that the mapping $\tilde u \mapsto u$ is a contraction and its fixed point is a solution $u$ to (\ref{eq_outgoing}) with the desired properties. Also in a similar manner one proves the continuous dependence of the solution on the initial data:
$$
\|\tilde u - u\|_{L^\infty_t L^{N+2}_x} \les \|\tilde u_0-u_0\|_{L^{N+2}} + R^{\frac{4-N}{N+2}} \|\tilde u - u\|_{L^\infty_t L^{N+2}_x} (\|\tilde u\|_{L^\infty_t L^{N+2}_x}^N + \|u\|_{L^\infty_t L^{N+2}_x}^N).
$$

Next, assume that $u_0 \in L^\infty$. Then
$$
\|u\|_{L^\infty_{t, x}} \les \|u_0\|_{L^\infty} + \Big\|\frac 1 r \int_0^r \rho |u(\rho, s)|^N u(\rho, s) \dd \rho\Big\|_{L^1_s L^\infty_x}.
$$
In the same manner as in the large data case we now obtain that
$$
\Big\|\frac 1 r \int_0^r \rho |u(\rho, s)|^N u(\rho, s) \dd \rho\Big\|_{L^\infty_x} \les (R+s)^{-\frac {2N+1}{N+2}} \|u(s)\|_{L^{N+2}}^{N+1}.
$$
For $-\frac {2N+1}{N+2}<-1$, i.e.\ $N>1$, $\int_0^\infty (R+s)^{-\frac {2N+1}{N+2}} \dd s = C R^{\frac {1-N}{N+2}} < \infty$. Therefore
$$
\|u\|_{L^\infty_{t, x}} \les \|u_0\|_{L^\infty} + R^{\frac {1-N}{N+2}} \|u\|_{L^\infty_t L^{N+2}_x}^{N+1} < \infty.
$$
Knowing this, it is easy to prove the persistence of regularity and of decay, albeit possibly with linear growth in the norms.
\end{proof}

\section{Conservation laws}\lb{conservation_laws}
We now state some conservation laws for equation (\ref{eq_outgoing}). The first refers to the conservation of the $L^{N+2}$ norm.

This is an a priori estimate. We cannot assume compact support of the solution due to the infinite speed of propagation; the best we can do is $\langle x \rangle^{-1}$ decay. The conditions in the statement are sufficient for all the integrals in the proof to be well-defined. In particular, we need to assume that $N>1$.

\begin{proposition}\lb{Nconserv} Suppose that $N>1$, that $u$ fulfills equation (\ref{eq_outgoing}) on $\R^3 \times I$, $I=[0, T]$, that $(u_0, u_1) \in \dot H^1 \times L^2$, and that uniformly for each $t \in [0, T]$ $u(t) \in \dot H^1 \cap \langle x \rangle^{-1} L^\infty$ and $u_t(t) \in L^2$. Then
\be\lb{conserv}\begin{aligned}
&\int_{\R^3 \times \{T\}} |u|^{N+2} \dd x = \int_{\R^3} |u_0|^{N+2} \dd x - N \int_0^T \int_{\R^3} \frac {|u|^{N+2}}{|x|} \dd x \dd t - \\
&\frac {N+2}{16\pi} \int_0^T \Big(\int_{\R^3} \frac {|u|^N u}{|x|} \dd x\Big)^2 \dd t + (N+2) \int_0^T \int_{\R^3} |u|^N u ((\partial_r+\frac 1 r)\Phi_0(t)(u_0, u_1)+ \\
&+\Phi_1(t)(u_0, u_1)) \dd x.
\end{aligned}\ee
In particular, if the initial data $(u_0, u_1)$ are outgoing, then
$$\begin{aligned}
\int_{\R^3 \times \{T\}} |u|^{N+2} \dd x =& \int_{\R^3} |u_0|^{N+2} \dd x - N \int_0^T \int_{\R^3} \frac {|u|^{N+2}}{|x|} \dd x \dd t - \\
&\frac {N+2}{16\pi} \int_0^T \Big(\int_{\R^3} \frac {|u|^N u}{|x|} \dd x\Big)^2 \dd t,
\end{aligned}$$
so
\be\lb{ln}
\|u\|_{L^\infty_t L^{N+2}_x(\R^3 \times [0, T])} \leq \|u_0\|_{L^{N+2}},\ \int_0^T \int_{\R^3} \frac {|u|^{N+2}}{|x|} \dd x \dd t \les \|u_0\|_{L^{N+2}}^{N+2}.
\ee
\end{proposition}
This estimate does not seem so useful because all the quantities involved are subcritical. However, as we saw above, if the solution is supported away from zero then the $L^{N+2}$ norm can be used to control it.
\begin{proof}[Proof of Proposition \ref{Nconserv}] As shown in Lemma \ref{lemma_firstorder}, (\ref{eq_outgoing}) is equivalent to (\ref{eqn}). Therefore
$$\begin{aligned}
&\frac d {dt} \int_{\R^3 \times \{t\}} |u|^{N+2} \dd x = (N+2) \int_{\R^3 \times \{t\}} |u|^Nuu_t \dd x \\
&= - (N+2)\int_{\R^3 \times \{t\}} |u|^Nuu_r + \frac {|u|^{N+2}}{|x|} + |u|^N u \frac 1 {2r} \Big(\int_0^r \rho |u(\rho)|^N u(\rho) \dd \rho\Big) \dd x + \\
&+ (N+2) \int_{\R^3 \times \{t\}} |u|^N u ((\partial_r+\frac 1 r)\Phi_0(t)(u_0, u_1)+\Phi_1(t)(u_0, u_1)) \dd x.
\end{aligned}$$
Furthermore, integrating by parts (see (\ref{parts})) we obtain that
$$
\int_{\R^3 \times\{t\}} (N+2) |u|^N u u_r \dd x = -2\int_{\R^3\times\{t\}} \frac {|u|^{N+2}}{|x|} \dd x.
$$
Also
$$\begin{aligned}
&\int_{\R^3 \times \{t\}} |u|^N u \frac 1 {2r} \Big(\int_0^r \rho |u(\rho)|^N u(\rho) \dd \rho\Big) \dd x = \\
&=2\pi \int_0^\infty r |u|^N u \Big(\int_0^r \rho |u(\rho)|^N u(\rho) \dd \rho\Big) \dd r \\
&= \pi \Big(\int_0^r \rho |u(\rho)|^N u(\rho) \dd \rho\Big)^2 \mid_0^\infty = \pi \Big(\int_0^\infty r |u(r)|^N u(r) \dd r\Big)^2 \\
&= \frac 1 {16 \pi} \Big(\int_{\R^3 \times \{t\}} \frac {|u|^N u}{|x|} \dd x\Big)^2.
\end{aligned}$$
Integrating from $0$ to $T$ we retrieve (\ref{conserv}).
\end{proof}

Next, we study the conservation of energy for (\ref{eq_outgoing}). For a solution $(u, v)$ of the equivalent system (\ref{eq_outgoing'}), let
$$
E_0[u(t)]:=\int_{\R^3 \times \{t\}} |\dl u|^2 + v^2 \dd x.
$$

Again, this is an a priori estimate and the conditions in its statement are sufficient for all the integrals that appear in the proof to be finite. Note again that we cannot assume that the solution has compact support due to the infinite speed of propagation. Consequently, we need that $N>1$. 

\begin{proposition}\lb{energy_bounds} Assume that $N>1$ and consider a solution $(u, v)$ of the system (\ref{eq_outgoing'}) on $\R^3 \times I$, $I=[0, T]$, such that $(u_0, u_1) \in \dot H^1 \times L^2$ and uniformly for each $t \in [0, T]$ $u(t) \in \dot H^2 \cap \dot H^1 \cap \langle x \rangle^{-1} L^\infty$ and $v \in \dot H^1 \cap L^2$. Then
\be\lb{e0}\begin{aligned}
E_0[u(T)]&=E_0[u(0)] - \frac {2N}{N+2} \int_0^T \int_{\R^3} \frac {|u|^{N+2}}{|x|} \dd x \dd t + \\
&+ \int_{\R^3 \times \{t\}} ((\partial_r+\frac 1 r) \Phi_0(t)(u_0, u_1) + \Phi_1(u_0, u_1)) |u|^N u \dd x.
\end{aligned}\ee
Moreover, if the initial data $(u_0, u_1)$ are outgoing,
\be\lb{e0'}
E_0[u(T)]=E_0[u(0)] - \frac {2N}{N+2} \int_0^T \int_{\R^3} \frac {|u|^{N+2}}{|x|} \dd x \dd t.
\ee
\end{proposition}
This shows that energy decreases with time for outgoing initial data.

Note that all computations can be justified under the weaker assumption that $u(t) \in \dot H^1_{rad} \cap L^\infty$ for each $t \in [0, T]$.
\begin{proof}[Proof of Proposition \ref{energy_bounds}]
We start from
$$\begin{aligned}
&\frac d {dt} \frac 1 2 \int_{\R^3 \times \{t\}} |\dl u|^2 + v^2 \dd x = \int_{\R^3 \times \{t\}} \dl u \cdot \dl u_t + v v_t \\
&= \int_{\R^3 \times \{t\}} \dl u \cdot \dl v - \dl u \cdot \dl \Big( \frac 1 {2r} \int_0^r \rho |u(\rho)|^N u(\rho) \dd \rho\Big) + v \Delta u + v \frac 1 2 |u|^N u \dd x \\
&= \int_{\R^3 \times \{t\}} \Big(-\frac 1 2 u_r |u|^N u\Big) + \frac 1 {2r^2} u_r \Big(\int_0^r \rho |u(\rho)|^N u(\rho) \dd \rho\Big) + \frac 1 2 u_t |u|^N u + \\
&+ \frac 1 {4r} \Big(\int_0^r \rho |u(\rho)|^N u(\rho) \dd \rho\Big) |u|^N u \dd x,
\end{aligned}$$
in view of the fact that
$$
\int_{\R^3 \times \{t\}} \dl u \cdot \dl v + v \Delta u \dd x = 0.
$$
We next look at each term individually.
$$
\begin{aligned}
&\int_{\R^3 \times\{t\}} \frac 1 2 u_r |u|^N u \dd x = 4\pi \int_0^\infty \frac 1 2 u_r |u|^N u r^2 \dd r \\
&= 4\pi \frac {|u|^{N+2}} {2(N+2)} r^2 \mid_0^\infty - 4\pi \int_0^\infty \frac {|u|^{N+2}}{N+2} r \dd r = -\frac 1 {N+2} \int_{\R^3\times \{t\}} \frac {|u|^{N+2}}{|x|} \dd x.
\end{aligned}$$
Then
$$
\begin{aligned}
&\int_{\R^3 \times\{t\}} \frac 1 {2r^2} u_r\Big(\int_0^r \rho |u(\rho)|^N u(\rho) \dd \rho \Big) \dd x = 4\pi \int_0^\infty \frac 1 2 u_r \Big(\int_0^r \rho |u(\rho)|^N u(\rho) \dd \rho \Big) \dd r \\
&= 2\pi u \Big(\int_0^r \rho |u(\rho)|^N u(\rho) \dd \rho \Big) \mid_0^\infty - 4\pi \int_0^\infty \frac 1 2 u r|u|^N u \dd r = -\int_{\R^3 \times \{t\}} \frac {|u|^{N+2}}{2|x|} \dd x.
\end{aligned}$$
Also
\be\lb{trei}
\int_{\R^3 \times \{t\}} \frac 1 2 u_t |u|^N u \dd x = \frac d {dt} \int_{\R^3 \times \{t\}} \frac{|u|^{N+2}}{2(N+2)} \dd x.
\ee
Finally,
\be\lb{patru}\begin{aligned}
&\int_{\R^3 \times \{t\}} \frac 1 {4r} \Big(\int_0^r \rho |u(\rho)|^N u(\rho) \dd \rho\Big) |u|^N u \dd x = 4\pi \int_0^\infty \frac 1 4 r |u|^N u \Big(\int_0^r \rho |u(\rho)|^N u(\rho)\Big) \dd \rho \\
&= \frac \pi 2 \Big(\int_0^r \rho |u(\rho)|^N u(\rho) \dd \rho\Big)^2 \mid_0^\infty = \frac \pi 2 \Big(\int_0^\infty r |u(r)|^N u(r) \dd r\Big)^2 = \frac 1 {32\pi} \Big(\int_{\R^3} \frac{|u|^N u}{|x|} \dd x\Big)^2.
\end{aligned}\ee
Therefore
$$\begin{aligned}
\frac d {dt} \frac 1 2 \int_{\R^3 \times \{t\}} |\dl u|^2 + v^2 \dd x & = - \frac N {2(N+2)} \int_{\R^3 \times \{t\}} \frac {|u|^{N+2}}{|x|} \dd x + \\
&+\frac d {dt} \int_{\R^3 \times \{t\}} \frac {|u|^{N+2}}{2(N+2)} \dd x + \frac 1 {32\pi} \Big(\int_{\R^3} \frac {|u|^N u}{|x|} \dd x\Big)^2.
\end{aligned}$$
Integrating from $0$ to $T$ we obtain an energy identity, but not the one we are looking for. 
For that, we use a different estimate instead of (\ref{trei}): 
by~(\ref{eqn})
$$
u_t+u_r+\frac u r+\frac 1 {2r} \Big(\int_0^r \rho |u(\rho)|^N u(\rho)\dd \rho\Big)=(\partial_r+\frac 1 r) \Phi_0(t)(u_0, u_1) + \Phi_1(u_0, u_1).
$$
The term (\ref{trei}) then becomes (see the computation (\ref{parts}))
$$\begin{aligned}
&\int_{\R^3\times \{t\}} \frac 1 2 u_t |u|^N u \dd x = - \frac N {2(N+2)} \int_{\R^3 \times \{t\}} \frac {|u|^{N+2}}{|x|} \dd x - \\
&\int_{\R^3 \times \{t\}} \frac 1 {4r} \Big(\int_0^r \rho |u(\rho)|^N u(\rho)\dd \rho\Big) |u|^N u \dd x + \frac 1 2 \int_{\R^3 \times \{t\}} ((\partial_r+\frac 1 r) \Phi_0(t)(u_0, u_1) + \Phi_1(u_0, u_1)) |u|^N u \dd x.
\end{aligned}$$
Among other things, this exactly cancels (\ref{patru}). In conclusion, by this method we get
$$\begin{aligned}
&\frac d {dt} \frac 1 2 \int_{\R^3 \times \{t\}} |\dl u|^2 + v^2 \dd x = - \frac N{N+2} \int_{\R^3 \times \{t\}} \frac {|u|^{N+2}}{|x|} \dd x + \\
&+ \frac 1 2 \int_{\R^3 \times \{t\}} ((\partial_r+\frac 1 r) \Phi_0(t)(u_0, u_1) + \Phi_1(u_0, u_1)) |u|^N u \dd x.
\end{aligned}$$
Integrating from $0$ to $T$ we obtain (\ref{e0}).
\end{proof}

\section{Proof of the main results}\lb{proof_main}

\begin{proof}[Proof of Theorem \ref{main_thm}] If the solution exists on some interval $[0, T]$, for all $t \geq T$ we can then rewrite the equation (\ref{eq_outgoing}) as
\be\lb{eq_outT}
u(t)=\Phi_0(t-T)(\tilde u_0, \tilde u_1) + \int_T^t \Phi_0(t-s)P_+(0, |u(s)|^Nu(s)) \dd s,
\ee
where
\be\lb{tildeu}
(\tilde u_0, \tilde u_1):=\Phi(T)(u_0, u_1)+\int_0^T\Phi(T-s)P_+(0, |u(s)|^Nu(s)) \dd s
\ee
are still outgoing (because the flow of the free wave equation, forward in time, preserves the outgoing property). Note that $\tilde u_0=u(T)$, but by taking a $T$ derivative we obtain
$$
u_t(T)=\tilde u_1-\frac 1 {2r}\int_0^r\rho|u(T, \rho)|^Nu(T, \rho) \dd \rho.
$$
Also compare with (\ref{eq_outgoing'}). The outgoing pair $(\tilde u_0, \tilde u_1)$ are, by (\ref{eq_outT}), the new initial data for the equation at time $T$.

By the existence result Proposition \ref{local_existence}, the solution $u$ exists at least locally in time, on some interval $[0, T_0]$ with
$$
T_0 = C \min(\|u_0\|_{\dot H^1}^{-2} \|u_0\|_{L^\infty}^{-N+2}, \|u_0\|_{L^\infty}^{-N}),
$$
and
$$
\|u\|_{L^\infty_t \dot H^1_x(\R^3 \times [0, T_0])} \les \|u_0\|_{\dot H^1},\ \|u\|_{L^\infty_{t, x}(\R^3 \times [0, T_0])} \les \|u_0\|_{L^\infty}.
$$

By approximating the initial data $u_0$ in the $\dot H^1 \cap L^\infty$ norm with $\dot H^1 \cap \dot H^2 \cap \langle x \rangle^{-1} L^\infty$ functions $\ov u_0$, we obtain approximating solutions $\ov u \in L^\infty_t (\dot H^2_x \cap \dot H^1_x \cap \langle x \rangle^{-1} L^\infty_x) (\R^3 \times [0, T_0])$ and such that $\ov v \in L^\infty_t (\dot H^1_x \cap L^2_x)(\R^3 \times [0, T_0])$, where $(\ov u, \ov v)$ satisfy (\ref{eq_outgoing'}). This is also shown in Proposition \ref{local_existence}.


For these smooth and decaying solutions $(\ov u, \ov v)$, Proposition \ref{energy_bounds} implies that energy is conserved and in particular
$$
\|\ov u(T_0)\|_{\dot H^1} \leq E_0[\ov u(T_0)]^{1/2} \leq E_0[\ov u(0)]^{1/2} \les \|\ov u_0\|_{\dot H^1}.
$$
We retrieve the same conclusion for the actual solution $u$ by passing to the limit, due to its continuous dependence on initial data.

Next, assume that the solution $u$ exists on the interval $[0, T_n]$ and has the desired property that
$$
\|u(T_n)\|_{\dot H^1} \les \|u_0\|_{\dot H^1}.
$$
The new initial data at time $T_n$ $(\tilde u_0 \equiv u(T_n), \tilde u_1)$ given by (\ref{tildeu}) are still outgoing, as stated above, and $\|\tilde u_0\|_{\dot H^1} \les \|u_0\|_{\dot H^1}$. 

In addition, due to the outgoing nature of the equation, $\tilde u_0$ and $\tilde u_1$ are supported on $\ov{\R^3 \setminus B(0, T_0)}$, hence by the radial Sobolev embedding
$$
\|\tilde u_0\|_{L^\infty} \les T_n^{-1/2} \|\tilde u_0\|_{\dot H^1}.
$$
By the existence result Proposition \ref{local_existence}, the solution $u$ can then be extended to the interval $[T_n, T_{n+1}=T_n + \delta T]$, where
$$
\delta T = C \min(\|\tilde u_0\|_{\dot H^1}^{-2} \|\tilde u_0\|_{L^\infty}^{-N+2}, \|\tilde u_0\|_{L^\infty}^{-N}) \\
\geq C \min(T_n^{(N-2)/2}, T_n^{N/2}) \|u_0\|_{\dot H^1}^N,
$$
and it has norm
$$
\|u\|_{L^\infty_t \dot H^1_x (\R^3 \times [T_n, T_{n+1}])} \les \|\tilde u_0\|_{\dot H^1},\ \|u\|_{L^\infty_{t, x} (\R^3 \times [T_n, T_{n+1}])} \les \|\tilde u_0\|_{L^\infty}.
$$
Furthermore, by the same approximation argument as above one can prove that $u$ obeys the energy conservation law on $[T_n, T_{n+1}]$ as well, hence
$$
\|u(T_{n+1})\|_{\dot H^1} \leq E_0[u(T_{n+1})]^{1/2} \leq E_0[u(0)]^{1/2} \les \|u_0\|_{\dot H^1}.
$$
This completes the induction step.

At any rate, $\delta T \ges \min(T_n^{(N-2)/2}, T_n^{N/2}) \geq \min(T_0^{(N-2)/2}, T_0^{N/2})$, so $\delta T$ is bounded from below, so $T_n \to \infty$ as $n \to \infty$. This proves the global existence of the solution. Concerning the norms, we see that
$$
\|u\|_{L^\infty_t \dot H^1_x} \les \|u_0\|_{\dot H^1}
$$
and
$$
\|u(t)\|_{L^\infty} \les t^{-1/2} \|u_0\|_{\dot H^1},\ \|u\|_{L^\infty_{t, x}(\R^3 \times [0, T_0=C\min(\|u_0\|_{\dot H^1}^{-2} \|u_0\|_{L^\infty}^{-N+2}, \|u_0\|_{L^\infty}^{-N})])} \les \|u_0\|_{L^\infty}.
$$
In particular, by combining the bounds on $[0, T_0]$ and on $[T_0, \infty)$ we obtain that
$$
\|u\|_{L^\infty_{t, x}} \les \|u_0\|_{L^\infty} + \|u_0\|_{\dot H^1}^{2} \|u_0\|_{L^\infty}^{(N-2)/2} + \|u_0\|_{\dot H^1} \|u_0\|_{L^\infty}^{N/2}.
$$
However, in this case nothing precludes the $L^2_t L^\infty_x$ Strichartz norm from being infinite.

When $T_n$ is sufficiently large and $N > 4$, the initial data $(\tilde u_0, \tilde u_1)$ at time $T_n$ become small in the sense of (\ref{small_data}), because
$$
\|\tilde u_0\|_{\dot H^1}^4 \|\tilde u_0\|_{L^\infty}^{N-4} + \|\tilde u_0\|_{\dot H^1}^2 \|\tilde u_0\|_{L^\infty}^{N-2} \les (T_n^{-(N-4)/2} + T_n^{-(N-2)/2}) \|u_0\|_{\dot H^1}^N.
$$
Therefore the solution $u$ exists globally on $\R^3 \times [T_N, \infty)$, with norm
$$\begin{aligned}
\|u\|_{L^\infty_t \dot H^1_x(\R^3 \times [T_n, \infty)) \cap L^2_t L^\infty_x(\R^3 \times [T_n, \infty))} &\les \|\tilde u_0\|_{\dot H^1} \les \|u_0\|_{\dot H^1},\\
\|u\|_{L^\infty_{t, x}(\R^3 \times [T_n, \infty))} &\les \|\tilde u_0\|_{L^\infty} \les T_n^{-1/2} \|u_0\|_{\dot H^1},
\end{aligned}$$
and we can stop the induction after finitely many steps.

Collecting the bounds we have obtained on each of the three intervals $[0, T_0]$, $[T_0, T_n]$, and $[T_n, \infty)$, we see that
$$\begin{aligned}
\|u\|_{L^2_t L^\infty_x} &\les \min(T_0, 1)^{1/2} \|u\|_{L^\infty_{t, x}(\R^3 \times [0, T_0])} + (\ln T_n - \ln \min(T_0, 1))^{1/2} \|u_0\|_{\dot H^1} + \\
&+ \|u\|_{L^2_t L^\infty_x (\R^3 \times [T_n, \infty))} \\
&\les \|u_0\|_{L^\infty} + (\ln_+ \|u_0\|_{\dot H^1} + \ln_+ \|u_0\|_{L^\infty} + 1)^{1/2} \|u_0\|_{\dot H^1}.
\end{aligned}$$
\end{proof}

\begin{proof}[Proof of Theorem \ref{Nexistence}] This is similar to the proof of Theorem \ref{main_thm}, but with the conservation law Proposition \ref{energy_bounds} replaced by Proposition \ref{Nconserv}.

First, again note that if a solution $u$ is defined on the interval $[0, T]$, then at time $T$ the solution solves the initial value problem (\ref{eq_outgoing}) with outgoing initial data $(\tilde u_0, \tilde u_1)$, where $\tilde u_0=u(T)$ and
\be\lb{u1}
\tilde u_1 = u_t(T) + \frac 1 {2r} \int_0^r \rho |u(T, \rho)|^N u(T, \rho) \dd \rho.
\ee

By the existence result Proposition \ref{LNexistence}, the solution $u$ exists at least on the interval $[0, T_0]$, with
$$
T_0 = C R^{\frac{2N-2}{N+2}}\|u_0\|^{-N}_{L^{N+2}}.
$$

We approximate the initial data $u_0$ by $\ov u_0 \in \dot H^1 \cap \langle x \rangle^{-1} L^\infty$ such that still $\supp \ov u_0 \subset \ov{\R^3 \setminus B(0, R)}$. This gives rise to approximating solutions $\ov u \in L^\infty_t \dot H^1_x(\R^3 \times I) \cap \langle x \rangle^{-1} L^\infty_{t, x}(\R^3 \times I)$.

For these smooth and decaying solutions, the conservation law Proposition \ref{Nconserv} holds, so $\|\ov u\|_{L^\infty_t L^{N+2}_x(\R^3 \times [0, T_0])} \leq \|\ov u_0\|_{L^{N+2}}$. By passing to the limit we also obtain that $\|u\|_{L^\infty_t L^{N+2}_x(\R^3 \times [0, T_0])} \leq \|u_0\|_{L^{N+2}}$.

Suppose that the solution $u$ exists on the interval $[0, T_n]$ and has the desired property that $\|u\|_{L^\infty_t L^{N+2}_x(\R^3 \times [0, T_n])} \leq \|u_0\|_{L^{N+2}}$. Due to the outgoing nature of the equation, $u(T_n)$ is supported on $\ov{\R^3 \setminus B(0, R+T_n)}$.

The new initial data at time $T_n$ $(\tilde u_0=u(T_n), \tilde u_1)$, with $\tilde u_1$ given by (\ref{u1}), then fulfill the conditions of Proposition \ref{LNexistence}. The solution $u$ can be extended to the interval $[T_n, T_{n+1}=T_n+\delta T]$, where
$$
\delta T = C (R+T_n)^{\frac{2N-2}{N+2}} \|u_0\|_{L^{N+2}}^{-N} \geq T_0.
$$
In addition, by the same approximation argument as above one can show that $\|u\|_{L^\infty_tL^{N+2}_x(\R^3 \times [T_n, T_{n+1}])} \leq \|u(T_n)\|_{L^{N+2}} \leq \|u_0\|_{L^{N+2}}$, thus completing the induction step.

Since $T_n \geq n T_0$, this proves the global existence of the solution $u$, satisfying (\ref{uN}), on $\R^3 \times [0, \infty)$.

When $N>4$, for sufficiently large $T_n$ the initial data at time $T_N$ become small in the sense of (\ref{LNsmall}), since $u(T_n)$ is supported on $\ov{\R^3\setminus B(0, R+T_n)}$ and condition (\ref{LNsmall}) becomes $$(R+T_n)^{\frac{4-N}{N+2}}\|u(T_n)\|_{L^{N+2}}^N \leq (R+T_n)^{\frac{4-N}{N+2}}\|u_0\|_{L^{N+2}}^N << 1.
$$

Therefore we can stop after finitely many induction steps and the solution exists globally on the interval $[T_n,\infty)$. One proves by dominated convergence that $\|u(t)\|_{L^{N+2}} \to 0$ as $t \to \infty$.
\end{proof}

\begin{proof}[Proof of Corollary \ref{optimal_existence}] By the local existence result Proposition \ref{well-posed}, the solution $u$ exists on some interval $[0, T]$ with $T>0$ and
$$
\|u\|_{L^\infty_t \dot H^{s_c}_x(\R^3 \times [0, T]) \cap L^{N/2}_t L^\infty_x(\R^3 \times [0, T])} \les \|u_0\|_{\dot H^s}.
$$

By adding the $L^{N+2}$ norm into the fixed point argument, one can show that if $u_0 \in L^{N+2}$ then $\|u\|_{L^\infty_t L^{N+2}_x(\R^3 \times [0, T])} \les \|u_0\|_{L^{N+2}}$. In fact, by approximating the solution $u$ with smooth and decaying solutions one can show that $u$ obeys the conservation law Proposition \ref{Nconserv} on $[0, T]$, so $\|u\|_{L^\infty_t L^{N+2}_x(\R^3 \times [0, T])} \leq \|u_0\|_{L^{N+2}}$.

Due to the outgoing nature of the equation, at time $T$ $u(T)$ is supported on $\ov{\R^3 \setminus B(0, T)}$. Then by Theorem \ref{Nexistence} the solution $u$ exists globally on $[T, \infty)$ and $\|u\|_{L^\infty_t L^{N+2}_x(\R^3 \times [T, \infty))} \leq \|u_0\|_{L^{N+2}}$. The conclusion follows.
\end{proof}

\section*{Acknowledgments}
We would like to thank Tom Spencer for the discussions we had on this topic.

M.B.\ was partially supported by the NSF grant DMS--1128155 and by an AMS--Simons Foundation travel grant.

A.S.\ is partially supported by the NSF grant DMS--1201394.


\begin{thebibliography}{CKSTT}
\bibitem[BeGo]{becgol} M.\ Beceanu, M.\ Goldberg, \emph{Strichartz estimates and maximal operators for the wave equation in $\R^3$}, Journal of Functional Analysis (2014), Vol.\ 266, Issue 3, pp.\ 1476--1510.
\bibitem[BeSo]{becsof} M.\ Beceanu, A.\ Soffer, \emph{Large outgoing solutions to supercritical wave equations}, preprint, arXiv:1601.06335.
\bibitem[BeL\"o]{bergh} J.\ Bergh, J.\ L\"ofstr\"om, \emph{Interpolation Spaces. An Introduction}, Springer-Verlag, 1976.
\bibitem[Bul1]{bul1} A.\  Bulut, \emph{The radial defocusing energy-supercritical cubic nonlinear wave equation in $\R^{1+5}$}, preprint, arXiv:1104.2002.
\bibitem[Bul2]{bul2} A.\ Bulut, \emph{Global well-posedness and scattering for the defocusing energy-supercritical cubic nonlinear wave equation}, preprint, arXiv:1006.4168.
\bibitem[Bul3]{bul3} A.\ Bulut, \emph{The defocusing energy-supercritical cubic nonlinear wave equation in dimension five}, preprint, arXiv:1112.0629.
\bibitem[DKM]{dkm} T.\ Duyckaerts, C.\ Kenig, F.\ Merle, \emph{Scattering for radial, bounded solutions of focusing supercritical wave equations}, preprint, arXiv:1208.2158.
\bibitem[GiVe]{give} J.\ Ginibre, G.\ Velo, \emph{Generalized Strichartz inequalities for the wave equation}, J.\ Func.\ Anal., 133 (1995), pp.\ 50--68.
\bibitem[KeTa]{keeltao} M.\ Keel, T.\ Tao, \emph{Endpoint Strichartz estimates}, American Journal of Mathematics (1998), Vol.\ 120, No.\ 5, pp.\ 955--980.
\bibitem[KiVi1]{kivi1} R.\ Kilip, M.\ Visan, \emph{The radial defocusing energy-supercritical nonlinear wave equation in all space dimensions}, preprint, arXiv:1002.1756.
\bibitem[KiVi2]{kivi2} R.\ Kilip, M.\ Visan, \emph{The defocusing energy-supercritical nonlinear wave equation in three space dimensions}, preprint, arXiv:1001.1761.
\bibitem[KlMa]{klma} S.\ Klainerman, M.\ Machedon, \emph{Space-time estimates for null forms and the local existence theorem}, Comm.\ Pure Appl.\ Math, 46 (1993), pp.\ 1221--1268.
\bibitem[KrSc]{krsc} J.\ Krieger, W.\ Schlag, \emph{Large global solutions for energy supercritical nonlinear wave equations on $\R^{3+1}$}, preprint, arXiv:1403.2913.
\bibitem[Roy1]{roy} T.\ Roy, \emph{Scattering above energy norm of solutions of a loglog energy-supercritical Schr\"{o}dinger equation with radial data}, preprint, arXiv:0911.0127.
\bibitem[Roy2]{roy2} T.\ Roy, \emph{Global existence of smooth solutions of a 3D loglog energy-supercritical wave equation}, preprint, arXiv:0810.5175.
\bibitem[Str]{struwe} M.\ Struwe, \emph{Global well-posedness of the Cauchy problem for a super-critical nonlinear wave equation in two space dimensions}, Mathematische Annalen 350.3 (2011), pp.\ 707--719.
\bibitem[Tao]{tao} T.\ Tao, \emph{Global regularity for a logarithmically supercritical defocusing nonlinear wave equation for spherically symmetric data}, J.\ Hyperbolic Diff.\ Eq., 4, 2007, pp.\ 259--266.
\end{thebibliography}
\end{document}